\documentclass[9pt]{extarticle}

\usepackage[utf8]{inputenc}

\usepackage{amsfonts, amsmath, amssymb, amsthm}
\usepackage{xfrac}
\usepackage{graphicx, xcolor, colortbl}
\usepackage{footnote}
\usepackage{tablefootnote}
\usepackage{natbib}
\usepackage{dsfont}
\usepackage{bbm}
\usepackage{algorithm}
\usepackage{algorithmic}
\usepackage{import}
\usepackage{mathtools}
\usepackage{bm}
\usepackage{xspace}
\usepackage{thmtools,thm-restate} 
\usepackage{accents}
\usepackage{framed}
\usepackage[inline]{enumitem}
\usepackage{booktabs}
\usepackage{color, colortbl}
\usepackage{natbib}
\usepackage{mathtools}

\usepackage{notation}

\newtheorem{lemma}{Lemma}
\newtheorem{theorem}{Theorem}
\newtheorem{corollary}{Corollary}

\theoremstyle{definition}

\newtheorem{proposition}{Proposition}

\newtheorem{remark}{Remark}

\usepackage{minitoc}

\usepackage{todonotes}

\title{Sharp Deviations Bounds for Dirichlet Weighted Sums with Application to analysis of Bayesian algorithms}
\author{Denis\,Belomestny, Pierre\,M\'enard, Alexey\,Naumov, Daniil\,Tiapkin, Michal\,Valko}
\date{January 2023}

\begin{document}

\maketitle

\doparttoc 
\faketableofcontents 

\begin{abstract}
In this work, we derive sharp non-asymptotic deviation bounds for weighted sums of Dirichlet random variables. These bounds are based on a novel integral representation of the density of a weighted Dirichlet sum. This representation allows us to obtain a Gaussian-like approximation for the sum distribution using geometry and complex analysis methods. Our results generalize similar bounds for the Beta distribution obtained in the seminal paper \cite{alfers1984normal}. Additionally, our results can be considered a sharp non-asymptotic version of the inverse of Sanov's theorem studied by \cite{ganesh1999inverse} in the Bayesian setting. Based on these results, we derive new deviation bounds for the Dirichlet process posterior means with application to Bayesian bootstrap. Finally, we apply our estimates to the analysis of the Multinomial Thompson Sampling (TS) algorithm in multi-armed bandits and significantly sharpen the existing regret bounds by making them independent of the size of the arms distribution support. 
\end{abstract}

\section{Introduction}
\label{sec:introduction}
One of the main multivariate distributions, the Dirichlet distribution is constrained to the simplex of a multidimensional space. This distribution has a lot of applications. The modeling of compositional data \citep{hijazi2009modelling}, parametric and non-parametric Bayesian statistics \citep{congdon2014applied,ghosal2017fundamentals}, topic modelling \citep{blei2003latent,teh2006hierarchical}, reinforcement learning \citep{osband2013more,osband2017why}, statistical genetics \citep{lange1995applications}, reliability \citep{somerville1997bayesian}, probabilistic constrained programming models \citep{dentcheva2006optimization}  are just a few examples where it can be found.  It possesses a variety of  probabilistic characterization properties, including conditional distributions, zero regression, and independence features.   Among these characteristics, we would like to draw attention to the fact that the Dirichlet distribution  is a conjugate prior of the parameters of the multinomial distribution in Bayesian statistics. 
\par
In this work, we are concerned with deviation bounds for weighted sums of Dirichlet random variables. Such weighted sums naturally appear in the analysis of Bayesian bootstrap methods \citep{rubin1981bayesian} as an approximation of the posterior mean. We aim to derive tight bounds on the crossing probabilities for such weighted sums, which optimally depend on the dimension of the underlying Dirichlet distribution, the number of summands. Generally, the computation of probabilities in multidimensional spaces, typically described by multiple integrals, is a particularly challenging topic in probability. The dimensional effect, sometimes known as the ``curse'' of dimensionality, is one of these problems defining characteristics that adds to its complexity. It states that each increase in dimension results in significant computational challenges. Given this phenomenon, any straightforward representation of and bound on multivariate probabilities can provide insightful data for theoretical research and real-world applications.  
\par
 The contribution of this paper is three-fold. First, we derive a novel integral representation for the density of a weighted sum of Dirichlet distributed random variables. This representation generalizes and sharpens the available representations; see Section~2 in \cite{ng2011dirichlet}. Second, based on this representation, we obtained a two-sided Gaussian-type bound for the deviations from a mean featuring the optimal dependence on the sum of Dirichlet parameters (``sample size'') and the dimension of the Dirichlet distribution. If this dimension equals $2$, then we arrive at the Beta distribution. In this respect, our results generalize the non-asymptotic bounds of \cite{alfers1984normal} to the case of Dirichlet distribution. Note that the bounds in \cite{alfers1984normal} are much more precise than those obtained from CLT-based approximations; see \cite{zubkov2013complete} for a comparative study. Applying our results to the Bayesian inference for measures on finite support, we obtain the non-asymptotic version of the so-called inverse of Sanov's theorem of \citet{ganesh1999inverse}. Third, we apply the derived estimates to analyze the Multinomial Thompson Sampling algorithm in bandits. The resulting instance-dependent regret bounds are much tighter than all previously known results in the literature; see \citet{riou20a}.  In particular, they feature an optimal leading term and a remaining term independent of the state space dimension.  
\par
The existing results on crossing probabilities (probabilities of the fixed size deviations from the mean)  for weighted Dirichlet sums could be more extensive in the literature. 
Let us first mention the exact formula for crossing probabilities in \cite{cho2001volume} for the case of Dirichlet distributions with equal parameters. Unfortunately, this formula is not very informative, and it seems complicated to derive proper bounds based on this result.   In \cite{baudry2021optimality}, some bounds on the Dirichlet crossing probabilities were obtained. Note that the lower bound (Corollary C.5.1) contains an additional exponential (in dimension) factor making this bound rough for large dimensions. This additional exponential factor comes from a product-like estimate for the density of the Dirichlet random vector. Applying the known probabilistic results (e.g., large deviation bounds or concentration inequalities) does not lead to the desired bound for several reasons. First, the components of the Dirichlet distribution are strongly dependent, making all probabilistic bounds for the sums of iid random variables inapplicable. Second, the known bounds for dependent random variables are rare in the literature and can not be used to obtain the sharp probabilistic boundary-crossing bounds we need. For example, in \cite{marchal2017sub} subgaussianity of the Dirichlet distribution was proved. However, the proxy variance in the corresponding subgaussian bound depends on the maximum of parameters of the underlying Dirichlet distribution leading to a rough estimate for the crossing probabilities.
  
 \paragraph*{Notations}
 Let $(\Xset,\Xsigma)$ be a measurable space and $\Pens(\Xset)$ be the set of all probability measures on this space. For $p \in \Pens(\Xset)$ we denote by $\E_p$ the expectation w.r.t. $p$. For random variable $\xi: \Xset \to \R$ notation $\xi \sim p$ means $\operatorname{Law}(\xi) = p$. We also write $\E_{\xi \sim p}$ instead of $\E_{p}$. For independent (resp. i.i.d.) random variables $\xi_\ell \mysim p_\ell$ (resp. $\xi_\ell \mysimiid p$), $\ell = 1, \ldots, d$, we will write $\E_{\xi_\ell \mysim p_\ell}$ (resp. $\E_{\xi_\ell \mysimiid p}$), to denote expectation w.r.t. product measure on $(\Xset^d, \Xsigma^{\otimes d})$. For any $p, q \in \Pens(\Xset)$ the Kullback-Leibler divergence $\KL(p, q)$ is given by
$$
\KL(p, q) = \begin{cases}
\E_{p}[\log \frac{\rmd p}{\rmd q}], & p \ll q, \\
+ \infty, & \text{otherwise}.
\end{cases} 
$$
For any $p \in \Pens(\Xset)$ and $f\colon \Xset \to \R$, $p f = \E_p[f]$. In particular, for any $p \in \simplex_d$ and $f\colon \{0, \ldots, d\}   \to  \R$, $pf =  \sum_{\ell = 0}^d f(\ell) p(\ell)$. Define $\Var_{p}(f) = \E_{s' \sim p} \big[(f(s')-p f)^2\big] = p[f^2] - (pf)^2$. 

For $d\in\N_{++}$, for any $\alpha=(\alpha_0,\ldots,\alpha_{d})\in\R_{++}^{d+1}$ we denote by $\Dir(\alpha)$ the Dirichlet distribution over the simplex $\Delta_d$ defined by density of the first $d$ components $p_\alpha(x_0,\ldots,x_{d-1}) = \prod_{i=0}^d x_i^{\alpha_i - 1}$ with a convention $x_{d} = 1 - \sum_{i=0}^{d-1} x_i$. In particular for $d=1$ the Dirichlet random vector has a form $(\xi, 1-\xi)$ for $\xi$ follows Beta distribution denoted by $\mathrm{Beta}(\alpha_0,\alpha_1)$.

\section{Main Results}
\label{sec:main_results}

Let $\cP[0,b]$ be the space of all probability measures supported on the segment $[0,b]$. Then we define the minimal Kullback-Leibler divergence for a measure $\nu \in \cP[0,b]$  and a real number $\mu\in [0,b],$
 \begin{eqnarray}
 \label{eq:Kinf}
\Kinf(\nu, \mu) = \inf\left\{  \KL(\nu, \eta): \eta \in \cP[0,b],\, \E_{X \sim \eta}[X] \geq \mu\right\}\,.
\end{eqnarray}
This quantity can be interpreted as a distance  from (projection of) the measure $\nu$ to the set of all measures with expectation at least $\mu$ where the distance is measured by the  KL-divergence. The measure $\eta$  solving the optimization problem  \eqref{eq:Kinf} is called moment projection ($M$-projection) or reversed information projection ($rI$-projection), see e.g.  \citep{csiszar2003information,bishop2006pattern} and \citep{murphy2022probabilistic}. Since KL-divergence is not symmetric, it is natural to compare this type of projection to a more common information projection ($I$-projection) 
 \[
    I(\nu , \mu) = \inf\{ \KL(\eta, \nu) : \eta \in \cP[0,b], \, \E_{X \sim \eta}[X] \geq \mu\}\,,
 \]
 that appears, for example, in Sanov-type deviation bounds \citep{sanov1961probability}.
The $I$-projections have an excellent geometric interpretation because $\KL$-divergence can be viewed as a Bregman divergence. The $M$-projections are not Bregman divergences and lack geometric interpretation. However, they are deeply connected to the maximum likelihood estimation when the measure $\nu$ is the empirical measure \citep[Lemma 3.1]{csiszar2004information} of a sample. Additionally, $M$-projections naturally appear as a rate function for a large deviation principle in a Bayesian framework \citep{ganesh1999inverse}, and also in lower bounds for the multi-armed bandit, see \citep{Lai1985asymptotically,burnetas1996optimal}. For an additional exposition on the multi-armed bandit,  see Section~\ref{sec:multinomial_ts}.
 \par
To simplify  notation in the sequel, we define the version of the minimal Kullback-Leibler distance for a finite support measures. Let us fix a function $f: \{0,\ldots,m\} \mapsto [0,b]$ and define for $p\in\simplex_{m}$  and $\mu\in\R,$
 \[
    \Kinf(p,\mu,f) = \inf\left\{  \KL(p,q): q\in\simplex_{m}, qf \geq \mu\right\}\,.
 \]
Next we present our main results on crossing probabilities. These results are summarized in the following theorem.
\begin{theorem}\label{thm:bound_dbc} 
For any $\alpha = (\alpha_0, \alpha_1, \ldots, \alpha_m) \in \R_{++}^{m+1}$ define  $\up \in \simplex_{m}$ with $\up(\ell) = \alpha_\ell/\ualpha, \ell = 0, \ldots, m,$ and $\ualpha = \sum_{j=0}^m \alpha_j$. Let $\varepsilon \in (0,1)$ and assume that $\alphamin = \min(\alpha_0, \alpha_m) \geq c_0 \cdot \varepsilon^{-2}$ for an absolute constant $c_0>0$ defined in \eqref{eq:c0_definition}. Let $f \colon \{0,\ldots,m\} \to [0,\ub]$ be any mapping  such that $f(0) = 0$, $f(m) = \ub.$ 
\begin{description}
\item [\textbf{Lower Bound}] Define $\alpha^+ = (\alpha_0, \alpha_1, \ldots, \alpha_{m-1}, \alpha_m+1),$ then  for any $\mu \in (\up f,  \ub),$ it holds 
\begin{eqnarray}
\label{eq:main-ineq-low}
 \P_{w \sim \Dir(\alpha^+)}(wf \geq \mu) \geq (1 - \varepsilon)\P_{\zeta \sim \cN(0,1)}\left(\zeta \geq \sqrt{2 \ualpha \Kinf(\up, \mu, f)}\right).
\end{eqnarray}
 \item [\textbf{Upper Bound}] Set $\alpha^- = (\alpha_0 + 1, \alpha_1, \ldots, \alpha_{m-1}, \alpha_m),$ then for any $\mu \in (\up f,  \ub),$ it holds 
\begin{eqnarray}
\label{eq:main-ineq-upper}
\P_{w \sim \Dir(\alpha^-)}(wf \geq \mu) \leq (1 + \varepsilon)\P_{\zeta \sim \cN(0,1)}\left(\zeta \geq \sqrt{2 \ualpha \Kinf(\up, \mu, f)}\right).
 \end{eqnarray}
  \end{description}  
\end{theorem}

Note that the theorem above holds only for $\mu \in (\up f, b)$. We can extend this theorem to any $\mu \in (0, b)$ by introducing an additional definition that is similar to one used by \cite{alfers1984normal} in the case of beta distribution. First, we define the quantity
\[
    \Kinfstar(p,\mu, f) = \inf\{  \KL(p,q) : q \in \simplex_m, qf = \mu\}
\]
that measures the reversed KL-transportation cost of moving a measure $p$ to a set of measures with expectation  $\mu$. In fact, for $\mu \geq pf,$ we have $\Kinfstar(p,\mu, f) = \Kinf(p, \mu, f)$ and  $\Kinfstar(p,\mu, f) = \Kinf(p, \ub-\mu, \ub - f)$ for $\mu < pf$. Then we define 
\[
    A(p, \mu, f) =  \mathrm{sgn}(\mu - pf) \cdot \sqrt{2 \Kinfstar(p,\mu, f)}.
\]
This function is a bijection between a segment $[0,b]$ and $\R$ for any non-degenerate $p \in \simplex_m$.
\begin{corollary}\label{cor:bound_dbc_all_u}
    For any $\alpha = (\alpha_0, \alpha_1, \ldots, \alpha_m) \in \R_{++}^{m+1}$ define  $\up \in \simplex_{m}$ with $\up(\ell) = \alpha_\ell/\ualpha, \ell = 0, \ldots, m,$ and $\ualpha = \sum_{j=0}^m \alpha_j$. Let $\varepsilon \in (0,1)$ and assume that $\alphamin = \min(\alpha_0, \alpha_m) \geq c_0 \cdot \varepsilon^{-2}$ for an absolute constant $c_0>0$ defined in \eqref{eq:c0_definition}. Let $f \colon \{0,\ldots,m\} \to [0,\ub]$ be any mapping  such that $f(0) = 0$, $f(m) = \ub.$ 
\begin{description}
\item [Lower Bound] Define $\alpha^+ = (\alpha_0, \alpha_1, \ldots, \alpha_{m-1}, \alpha_m+1),$ then  for any $\mu \in (0,  \ub),$ it holds 
\begin{eqnarray}
\label{eq:main-ineq-low_all_u}
 \P_{w \sim \Dir(\alpha^+)}(wf \geq \mu) \geq (1 - \varepsilon)\P_{\zeta \sim \cN(0,1)}\left(\zeta \geq \sqrt{\ualpha} \cdot A(\up, \mu, f)\right).
\end{eqnarray}
 \item [Upper Bound] Set $\alpha^- = (\alpha_0 + 1, \alpha_1, \ldots, \alpha_{m-1}, \alpha_m),$ then for any $\mu \in (0,  \ub)$ it holds 
\begin{eqnarray}
\label{eq:main-ineq-upper_all_u}
\P_{w \sim \Dir(\alpha^-)}(wf \geq \mu) \leq (1 + \varepsilon)\P_{\zeta \sim \cN(0,1)}\left(\zeta \geq \sqrt{\ualpha} \cdot A(\up, \mu, f)\right).
 \end{eqnarray}
  \end{description}  
\end{corollary}

\begin{corollary}\label{cor:two_sided_ineeq}
    For any $\alpha = (\alpha_0,\ldots,\alpha_m) \in \R^{m+1}_{++}$ define $\ualpha = \sum_{i=0}^m \alpha_i$ and two distributions  $\up^+ \in \simplex_m,$ $\up^{-} \in \simplex_m$ with 
    \begin{eqnarray*}
     \up^+(\ell) &= & \alpha_\ell / ( \ualpha-1),\quad \ell < m, \quad \up^{+}(m) = (\alpha_m - 1)/ (\ualpha-1),
    \\
    \up^-(\ell) &= & \alpha_\ell / (\ualpha - 1), \quad \ell > 0,\quad \up^-(0) = (\alpha_0 - 1)/(\ualpha - 1).   
    \end{eqnarray*}
    Assume that $\alphamin = \min(\alpha_0, \alpha_m) \geq c_0 \cdot \varepsilon^{-2} + 1$ for an absolute constant $c_0 > 0$. Then the following two-sided bound holds for any $\mu \in (0, \ub)$
    \[
        (1 - \varepsilon)\P\Bigl(\zeta \geq \sqrt{\ualpha-1} \cdot  A(\up^+, \mu, f)\Bigr) \leq \P(wf \geq \mu)   \leq (1 + \varepsilon)\P\Bigl(\zeta \geq \sqrt{\ualpha-1} \cdot A(\up^-, \mu, f)\Bigr),
    \]
    where $w \sim \Dir(\alpha)$ and $\zeta \sim \cN(0,1)$.
\end{corollary}

\paragraph*{Discussion}
Let us discuss conditions of the theorem. 
\begin{itemize}
    \item The condition on $\alphamin$  is needed to bound  remainder terms in the asymptotic expansion for  the integral representation for the density of \(wf\), see Section~\ref{sec:saddle-point}. In the Bayesian setting, this condition could be automatically satisfied by taking large enough uniform prior.
    \item The condition on the weight function $f$ can be achieved by an appropriate shifting of $u$ and $f$ by the same constant;
    \item The usage of $\alpha^+$ and $\alpha^-$ is equivalent to the  shifts in parameters that were used by \citet{alfers1984normal}. In Corollary~\ref{cor:two_sided_ineeq} we have that the upper and lower bound are close to each other as $\ualpha \to +\infty$ since $\norm{\up^+ - \up^-}_1 = 2/\ualpha$ and Theorem~7 by \cite{honda2010asymptotically} holds.
\end{itemize}
Let us compare the bounds of Theorem~\ref{thm:bound_dbc} to the existing bounds in the literature. 
First, note that the components of the Dirichlet vector $\xi\sim \Dir(\alpha)$ are strongly dependent. In particular, 
\[  
\cov[\xi_i,\xi_j]=-\frac{\alpha_i\alpha_j}{\ualpha^2(1+\ualpha)}<0,\quad i\neq j.
\]
This implies for the variance of \(\xi f,\) 
\[
\Var(\xi f)=\sum_{i\neq j}\frac{f_{i}f_{j}\alpha_{i}\alpha_{j}}{\overline{\alpha}^{2}(1+\overline{\alpha})}+\sum_{i}\frac{(\overline{\alpha}-\alpha_{i})\alpha_{i}f_{i}^{2}}{\overline{\alpha}^{2}(1+\overline{\alpha})}.
\]
If the components of $\xi$ were independent we could apply Gaussian approximation results (see e.g. \cite{fang2021high} and references therein) to 
$(\xi f-\up f)/\sqrt{\Var(\xi f)}$ and get bounds for 
\begin{eqnarray*}
\P_{ \xi \sim\Dir(\alpha)}(\xi f \geq \mu) =\P_{ \xi\sim \Dir(\alpha)}\Bigl((\xi f-\up f)/\sqrt{\Var(\xi f)}\geq 
(\mu-\up f)/\sqrt{\Var(\xi f)}\Bigr).
\end{eqnarray*}
 However, even in the independent case, obtaining two-sided bounds of multiplicative form involving  \(\Kinf\) as a deviation measure remains unclear. Another possibility to get bounds of the form \eqref{eq:main-ineq-low} and \eqref{eq:main-ineq-upper} is to use  large deviations techniques (Sanov-type bounds). Apart from the fact that there are only a few results in the literature for dependent case, this technique leads to asymptotic results and hence can not be directly used to obtain the nonasymptotic bounds we need. For example, the so-called inversed Sanov-type large deviation principle was established by \cite{ganesh1999inverse}. The authors in \cite{ganesh1999inverse} consider a problem of Bayesian inference for a finite-support distribution $p \in \simplex_m$ over a space of size $m+1$. Taking a prior distribution $\rho^0$ to be a Dirichlet distribution $\Dir(\alpha^0)$, the posterior distribution $\rho^n$ is also Dirichlet distribution $\Dir(\alpha^n)$ with  $\sum_{j=0}^m \alpha^n_j \to \infty$. Therefore, Theorem~1 in \cite{ganesh1999inverse} implies that for sets $B$ of the form $B = \{ w : wf \geq \mu \}$ with $\mu \geq pf,$ the following asymptotic bound holds
\[
    \lim_{n \to \infty} \frac{1}{n} \log \P_{w \sim \rho^n}[ wf  \geq \mu ] = \inf_{q: qf \geq \mu } \KL(p, q) = \Kinf(p, \mu, f).
\]
Therefore, we see that our Theorem~\ref{thm:bound_dbc} can be viewed as a non-asymptotic version of this large-deviation principle for a Bayesian inference under mild conditions on the prior distribution. Finally, let us mention a connection to the literature on deviations bounds for a weighted bootstrap procedure, see e.g. \cite{broniatowski2017weighted} and references therein. These works provide Sanov-type large deviation results for the bootstrapped empirical measure of the points \(x_1,\ldots,x_n,\)
 \begin{eqnarray*}
P_n^W=\sum_{i=1}^n W_i^n\delta_{x_i}, \quad W^n_i=\frac{Y_i}{\sum_{i=1}^n Y_i}
\end{eqnarray*}
where \(Y_1,\ldots,Y_n\) denotes a sequence of nonnegative independent real-valued random variables with expectation \(1\). If we chose Gamma distribution for \(Y,\) then the weights become Dirichlet distributed.  
The main problem with this approach is that only asymptotic results can be obtained in this way with remainder terms that are difficult to quantify.

\section{Application: Deviations for the Dirichlet process posterior}

Consider a space $\Xset$ and its $\sigma$-algebra $\mathcal{X}$, and define $\nu$ as a finite non-null measure on $(\Xset,\mathcal{X})$. The stochastic process $G$, indexed by elements $B$ of $\mathcal{X}$, is a Dirichlet Process with parameter $\nu$ ($G\sim\mathrm{DP}(\nu)$) if
\begin{equation*}
	G\left(B_1\right),\ldots,G\left(B_d\right)\sim\mathrm{Dir}\left(\nu\left(B_1\right), \ldots, \nu\left(B_d\right)\right),
\end{equation*}
for any measurable partition $\left(B_1, \ldots, B_d\right)$ of $\Xset$, see \cite{ferguson1973bayesian}. 
\par
Let \(\widehat{P}_n = n^{-1}\sum_{i=1}^n \delta_{Z_i}\) be the empirical measure of an i.i.d. sample \(Z_1,\ldots, Z_n\) from a distribution \(P\) on  \((\Xset,\mathcal{X})\), and given \(Z_1,\ldots,Z_n\) let \(\widetilde{P}_n\) be drawn from the Dirichlet process with a base measure \(\nu + n \widehat{P}_n\). Here \(\nu\) is a finite (not necessarily probability) measure on the sample space and \(\widetilde{P}_n|Z_1,\ldots,Z_n \sim \mathrm{DP}(\nu + n \widehat{P}_n)\) for all \(n\), which is the posterior distribution obtained when equipping the distribution of the observations \(Z_1, Z_2,\ldots, Z_n\) with a Dirichlet process prior with the base measure \(\nu\). For full definitions and properties, see the review in Chapter 4 of \cite{ghosal2017fundamentals}. 
\par
We will be interested in deviations bounds for the process \(\widetilde{P}_ng\) for some bounded function \(g\) on $\Xset$. Notably, the following representation holds 
\begin{equation}\label{eq:dirichlet_process}
    \widetilde{P}_n g=V_nQg+(1-V_n)\frac{\sum_{i=1}^n W_ig(Z_i)}{\sum_{i=1}^n W_i}
\end{equation}
with \(V_n\sim \mathrm{Beta}(|\nu|,n),\) \(Q\sim \mathrm{DP}(\nu),\) and \(W_1,W_2,\ldots\) are iid exponential variables with mean \(1\). All variables \(V_n,Q,W_1,W_2,\ldots\) are independent. For a proof, see, e.g. Theorem 14.37 of \cite{ghosal2017fundamentals} with $\sigma = 0$. Some deviation bounds of asymptotic form (large deviation principle)  can be found in \citet{ayalvadi2000large}. 
\par
Let $g$ be a bounded function with values in $[0,1]$ and assume that we know $x_0, x_1$ such that $g(x_0) = 0$ and $g(x_1) = 1$. In particular, it could be guaranteed by adding these two points to the set $\Xset$. Then  introduce the base measure $\nu_\gamma = \gamma \delta_{x_0} + \gamma \delta_{x_1}$ where $\gamma>0$ is a fixed number. Then $Q \sim \mathrm{DP}(\nu_\gamma)$ corresponds to the  Dirichlet distribution supported on $x_0$ and $x_1$:
\[
    Q = w_0 \cdot \delta_{x_0} + w_1 \cdot \delta_{x_1}, \qquad (w_0, w_1) \sim \Dir(\gamma, \gamma),
\]
and, in particular, the posterior mean of $g$  is given by
\[
    \widetilde{P}_n g = \sum_{i=0}^{n+1} w_i \cdot g(Z_i), \qquad w \sim \Dir(\gamma, 1, \ldots, 1, \gamma),
\]
where $Z_0= x_0$ and $Z_{n+1} = x_1$ by convention. Notably, we see that the Dirichlet process with a finite-support base measure as a prior leads to the Bayesian bootstrap \citep{rubin1981bayesian}. By Corollary~\ref{cor:two_sided_ineeq}, we have that for any $\varepsilon > 0,$ under the choice $\gamma = c_0 \cdot \varepsilon^{-2} + 1$  the following deviation inequality holds for any $t \in (0, 1-Pg)$ conditionally on $Z = (Z_1,\ldots,Z_n),$
\[
    \P(\widetilde{P}_n g - \widehat{P}_n g \geq t | Z)  \leq (1 + \varepsilon)\P\left(\zeta \geq \sqrt{2 (n + 2\gamma - 1) \Kinf\left( \overline{\nu}_{n,g}, \widehat{P}_n g + t\right)}\right),
\]
where $\overline{\nu}_{n,g}$ is a posterior measure over $[0,1]$ defined as 
\[
    \overline{\nu}_{n,g} = \frac{\gamma-1}{n+2\gamma-1} \delta_0 + \frac{1}{n+2\gamma-1}\sum_{i=1}^{n} \delta_{g(Z_i)} + \frac{\gamma}{n+2\gamma-1} \delta_1,
\]
and we treat $\Kinf$ as the minimal Kullback-Leibler divergence over all measures supported on $[0,1]$. By Lemma~11 of \cite{garivier2018kl},  we have
\begin{align*}
    \Kinf\left( \overline{\nu}_{n,g}, \widehat{P}_n g + t\right) \geq  2\left( \frac{2\gamma - 1}{n+2\gamma-1}\left[\widehat{P}_n g - \frac{\gamma}{2\gamma - 1}\right] + t \right)_+^2\,
\end{align*}
and this implies a Hoeffding-type inequality for the posterior Dirichlet process.
\begin{proposition}[Hoeffding-type inequality for Dirichlet process]
    Let $Z_1,\ldots, Z_n$ be sample from a distribution $P$ on some measurable space $(\Xset, \cX)$ and let $g \colon \Xset \to [0,1]$ be a function satisfying  $g(x_0) = 0, g(x_1) = 1$ for some $x_0, x_1 \in \Xset$.
    Furthermore, let $\widetilde{P}_n$ be drawn from the posterior distribution $\mathrm{DP}(\nu_\gamma + n \widehat{P}_n)$ with a base measure $\nu_\gamma = \gamma \delta_{x_0} + \gamma \delta_{x_1}$. Then for any $\delta \in (0,1),$  $\gamma \geq c_0 \varepsilon^{-2} + 1$ and a fixed $\varepsilon \in (0,1),$ we have
    \[
    \P\left(\widetilde{P}_n g -  \widehat{P}_n g \geq \sqrt{\frac{\log((1+\varepsilon)/\delta)}{2(n+2\gamma-1)}} + \frac{\gamma}{n+2\gamma - 1}\  \biggl| \ Z\right)  \leq \delta.
\]
\end{proposition}
Moreover, using a slightly more involved technique, we can obtain a Bernstein-type inequality.
\begin{proposition}[Bernstein-type inequality for Dirichlet process]    
    Let $\widetilde{P}_n$ be drawn from the posterior distribution $\mathrm{DP}(\nu_\gamma + n \widehat{P}_n)$ for a base measure $\nu_\gamma = \gamma \delta_{x_0} + \gamma \delta_{x_1}$. Then for any $\delta \in (0,1),$  $\gamma \geq c_0 \varepsilon^{-2} + 1$ and a fixed $\varepsilon \in (0,1),$ we have
        \[
        \P\left(\widetilde{P}_n g -  \widehat{P}_n g \geq \sqrt{\frac{4  \Var_{\widehat{P}_n}[g] \cdot \log((1+\varepsilon)/\delta)}{n+2\gamma-1}} + \frac{4 \log((1+\varepsilon)/\delta) + 5\gamma}{n+2\gamma - 1}\  \biggl| \ Z\right)  \leq \delta\,,
    \]
    where $\Var_{\widehat{P}_n}[g] = \widehat{P}_n[g^2] - [\widehat{P}_n g]^2$ is the variance of the empirical measure $\widehat{P}_n$.
\end{proposition}
\begin{proof}
    By continuity of $\Kinf(\overline{\nu}_{n,g}, \mu)$ in the second argument, we have that for any $\delta \in (0,1),$ there is $\mu_\delta$ such that
    \[
        \Kinf(\overline{\nu}_{n,g}, \mu_\delta) = \frac{\log((1+\varepsilon)/\delta)}{n+2\gamma - 1}.
    \]
    Let $\eta$ be a measure such that $\Kinf(\overline{\nu}_{n,g}, \mu_\delta) = \KL(\overline{\nu}_{n,g}, \eta)$. By Lemma 9 of \cite{honda2010asymptotically}, this measure  has a finite support, that coincides with the one of  measure $\overline{\nu}_{n,g}$  since $\overline{\nu}_{n,g}\{1\} > 0$. Therefore Corollary 11 of \cite{talebi2018variance} is applicable and we get
    \[
          \E_{X \sim \eta}[X]  - \E_{X \sim \overline{\nu}_{n,g}}[X] \leq \sqrt{2 \Var_{X \sim \eta}[X] \KL(\overline{\nu}_{n,g}, \eta)}.
    \]
    By a change of variance argument (see e.g. Lemma~E.3 and Lemma~E.4 by \cite{tiapkin2022dirichlet}),
    \[
        \Var_{X \sim \eta}[X] \leq 2 \Var_{X \sim \overline{\nu}_{n,g}}[X] + 4 \KL(\overline{\nu}_{n,g}, \eta) \leq 2 \Var_{\widehat{P}_n}[g] + \frac{3(2\gamma - 1)}{n+2\gamma - 1} + 4 \KL(\overline{\nu}_{n,g}, \eta)
    \]
    we derive
    \[
        \E_{X \sim \eta}[X] \leq \widehat{P}_n +  \sqrt{ \frac{4 \Var_{X \sim \overline{\nu}_{n,g}}[X] \cdot\log((1+\varepsilon)/\delta)}{n+2\gamma - 1}} + \frac{ (\sqrt{8} + 1) \log((1+\varepsilon)/\delta)  + 4\gamma}{n+2\gamma - 1},
    \]
    and, as a result
    \[
        \P\left(\widetilde{P}_n g -  \widehat{P}_n g \geq \sqrt{\frac{4  \Var_{\widehat{P}_n}[g] \cdot \log((1+\varepsilon)/\delta)}{n+2\gamma-1}} + \frac{4 \log((1+\varepsilon)/\delta) + 4\gamma}{n+2\gamma - 1}\  \biggl| \ Z\right)  \leq \delta.
    \]
\end{proof}

Let us stress that our bound could also be used in a similar manner to directly estimate the difference between the corresponding posterior mean and the real mean.
\section{Application: Refined analysis of Multinomial Thompson Sampling}\label{sec:multinomial_ts}

In this section, we will apply Theorem~\ref{thm:bound_dbc} to analyze an algorithm for the stochastic $K$-armed bandit problem. In this sequential problem, an agent, samples one arm (or distribution) out of the $K$ independent arms at each round and receives the corresponding sample as a reward for an horizon of $T$ rounds. The regret, which measures the difference between the cumulative reward collected by an oracle that knows the mean of each arm in advance and the cumulative reward collected by the agent, is one of the most studied performance criteria for this problem.

In their seminal work, \citet{Lai1985asymptotically} provides an asymptotic problem-dependent lower bound for the parametric bandit, where the arms belong to the same parametric family of distributions. This lower bound states that any good algorithm should suffer a regret of at least $\kappa \log(T)$ as $T$ grows to infinity and where $\kappa$ is some informational complexity measure of the problem. Since then, many algorithms have been shown to be asymptotically optimal, matching the lower bound by \citet{Lai1985asymptotically}. Notably, the famous upper confidence bound (\klUCB) algorithm \citep{agrawal1995sample,auer2002nonstochastic} is asymptotically optimal for a wide range of exponential families distribution \citep{agrawal1995sample,burnetas1996optimal,garivier2011kl}. Additionally, the Thompson sampling (\TS , \citealp{thompson1933on}) algorithm, which exhibits good empirical performance \citep{chapelle2011empirical}, is also proven to be asymptotically optimal \citep{korda2013thompson, agrawal2013further}.

In this paper, we study the more challenging setting of non-parametric stochastic bandit where we only assume that the arms are supported on the unit interval, i.e that the rewards are bounded. This setting was first considered by \citet{honda2010asymptotically}. In particular, they show that the \DMED algorithm matches  the asymptotic problem-dependent lower bound by \citet{Lai1985asymptotically} generalized by \citet{burnetas1996optimal} to the non-parametric setting of bounded rewards (see also \citet{honda2015nonasymptotic}). Later \citet{cappe2013kullback, garivier2018kl}, leveraging the empirical likelihood method, proposed an extension of the \klUCB algorithm, namely \KLUCB, for bandit with bounded rewards that is also asymptotically optimal.
\par
In their recent work,  \citet{riou20a} proposed an extension of \TS for bounded rewards, called non-parametric Thompson sampling (\NPTS). This algorithm computes an average of the observed reward with random weights for all arms at each round, and selects the arm with the highest average. See \citet{baudry2021optimality} for an extension of this algorithm beyond the bounded reward setting. While \NPTS does not rely on the posterior sampling mechanism a priori, \citet{riou20a} claim that \NPTS is still a randomized algorithm, but not a \TS algorithm in the strict sense. 
However, the representation \eqref{eq:dirichlet_process} for Dirichlet process with $\Xset = [0,1]$ and $g(x) = x$ shows that \NPTS is indeed a \TS-type algorithm with a Bayesian model with prior a Dirichlet Process $\mathrm{DP}(\delta_1)$ of base measure a Dirac distribution.
\par
Nevertheless, dealing with the non-parametric setting comes at the cost of considerable time and space complexity. Indeed all the presented algorithms require storing all the observed rewards and for \NPTS sample weights for all the history or for \KLUCB and \DMED solve a convex program at each round, ending in a space complexity of at least $\cO(t)$ and time-complexity $\cO(t)$ at round $t$. This fact contrasts sharply with the parametric setting where the time complexity is of order $\cO(1)$ and space complexity is of order $\cO(K)$ at all rounds.
\par
Interestingly the paper \citet{riou20a} also considers the multinomial reward setting where the support of the arms is the same finite set of size $m$. We can think of this setting as lying at the boundary between non-parametric and parametric bandit problems.
Note that this setting was previously considered by \citet{cappe2013kullback}. For this setting, \citet{riou20a} propose the Multinomial Thompson Sampling (\MTS) algorithm that turns out to be a \TS algorithm with a Dirichlet prior/posterior over the arms.
\par
The current paper derives a refined instance-dependent regret bound for \MTS. Notably, our bound is independent of the size of the finite support of the arms. This property is in striking contrast with the previous regret bound for \MTS, which grows exponentially with the support size. Furthermore, this improvement shows that an extension of \MTS based on a randomized rounding technique is optimal. However, more importantly, this extension enjoys a space complexity of order $\cO(K\log(T))$ and a per-round time-complexity of order $\cO(\log(T))$.

Let us start from the description of the setting. We consider a bandit $\bnu= (\nu_1,\ldots,\nu_K)$ with $K$ arms where the arms belong to some set $\cV$ of probability distributions $\nu_k\in \cV, \forall k\in[K]$. At each round $t\in[T]$, the agent selects one of the arms $A^t \in [K]$ and receives a random reward $Y_t$ drawn independently at random from the distribution $\nu_{A^t}$. A typical performance criterion for the agent is the expected regret. 
\[
    \regret^T = \E\left[ \sum_{t=1}^T (\mustar - \mu_{A^t}) \right] = \sum_{a=1}^K (\mustar-\mu_a)\E[N_a^t], 
\]
where $\mu_a = \E_{Y\sim \nu_a}[Y]$ is the mean reward of arm $a$, $\mustar=\max_{a\in[K]} \mu_a$ is the mean reward of an optimal arm and $N_a^t = \sum_{s=1}^t \ind\{A^s=a\}$
the number of
draws of arm $a$ at the end of round $t$. For the sake of simplicity we assume that there is a unique optimal arm $a^\star: \mustar = \mu_{a^\star}$.
\par
We now describe the problem-dependent lower bound by \citet{Lai1985asymptotically}, see also \citet{burnetas1996optimal} and \citet{garivier2016explore}.
 A key quantity in this lower bound is the so-called minimal Kullback-Leibler divergence introduced in Section~\ref{sec:main_results}; for arm $a$, and  $\mu\in\R$
 \[\Kinf^{\cV}(\nu_a,\mu) = 
 \inf\left\{  \KL(\nu, \eta): \eta \in \cV,\, \E_{X \sim \eta}[X] \geq \mu\right\}\,.\]
Then for any "reasonable" agent\footnote{See \citet{garivier2016explore} for a formal definition.}, for any bandit problem $\bnu,$ it holds
 \[ \liminf_{T\to\infty} \frac{\regret^T}{\log T} \geq \sum_{a: \Delta_a > 0} \frac{\Delta_a}{\Kinf^{\cV}(\nu_a, \mustar)}\]
where $\Delta_a = \mustar - \mu_a$ is the sub-optimality gap of arm $a$. In the sequel, we are interested in  problem-dependent optimal algorithms for which the reverse inequality
holds with a lim sup instead of a lim inf. We focus on two particular settings: the multinomial reward setting where each arm has a distribution with finite support and the bounded reward setting.

\subsection{Multinomial reward}

In this section we assume that the arm are supported on the set $\{0, 1/m, \ldots, 1\}$ for a fixed $m$. In particular, we have $\cV = \Pens(\{0, 1/m, \ldots, 1\})$ and denote the minimal Kullback divergence by $\Kinf^{(m)}(\nu_a,\mu)=\Kinf^\cV(\nu_a,\mu) = \Kinf(p_a,\mu,f)$ where $p_a \in \simplex_m$ is defined as $p_a(i) = \nu_a\{i/m\}$ for $i \in \{0,\ldots,m\}$, and  $f(x) = x/m$. 
\par
We now describe the \MTS algorithm introduced by \citet{riou20a} but we first need to present the associated Bayesian model.
As the rewards are sampled from a multinomial distribution, it is appropriate to use a Dirichlet distribution as the prior, as it is a conjugate prior to the multinomial distribution. Thus, let $\rho^0_a=\Dir(\alpha^0_a)$ be a Dirichlet prior distribution with parameter $\alpha^0_a \in \R_+^{m+1}$ over the simplex $\simplex_m$ for arm $a$. The posterior after $t$ rounds is then a Dirichlet distribution, denoted by $\rho^t_a=\Dir(\alpha^t_a)$, with parameter $\alpha^t_a(i) = \alpha^0_a(i) + \sum_{s=1}^t \ind\{A_s=a,Y_s=i/m\}$ for all categories $i = 0,\ldots,m$. At round $t$, \MTS generates posterior samples $w^t_a \sim \rho^{t-1}_a$ for all arms $a \in [K]$ and selects the arm $A^{t} = \argmax_{a \in [K]} w^t_a f$ with the highest posterior sample mean, see Algorithm~\ref{alg:MTS} in Appendix~\ref{app:mts_algorithms} for a full description.

\citet{riou20a} showed that this algorithm has optimal problem-dependent regret for large enough $T$. However, their analysis is, in fact, very asymptotic since the second-order terms of the regret bound depends exponentially on the size of the support $m$. This implies that $T$ should be of order $\rme^{\rme^m}$ to get the correct (non-asymptotic) dependence on $T$. 
Using  Theorem~\ref{thm:bound_dbc}  with a careful choice of the prior distribution, we prove the following theorem in Appendix~\ref{app:proof_bound_mts} that features a much sharper dependence on $m$.

\begin{theorem}\label{thm:multinomial_ts_regret}
    For \MTS with Dirichlet prior parameter given by $\alpha^0_0 = \alpha^0_m = 4c_0 + 1$, where $c_0$ is a constant defined in Theorem~\ref{thm:bound_dbc} and $\alpha^0_i = 1/(m-2)$ for all $i \in \{1,\ldots, m-1\}$, for any sub-optimal arm $a$
    \[
        \E[N^T_a] \leq  \frac{\log T}{\Kinf^{(m)}(\nu_a,\mustar)}  + \cO\left( \log^{9/10}(T) \right)
    \]
    and consequently
    \[
        \regret^T \leq \sum_{a: \Delta_a > 0} \frac{\Delta_a \log(T)}{\Kinf^{(m)}(\nu_a,\mustar)} + \cO\left( \log^{9/10}(T) \right).
    \]
    In particular, the regret bound does not depend on the size of support $m$.
\end{theorem}

Note that the previous analysis by \citet{riou20a} (see also \citealp{baudry2023general}) relies on a deviation inequality for weighted sum of Dirichlet random variables of the form 
$\P(w_a^t f \geq \mu) \gtrsim n^{-m} \exp(-n \Kinf^{(m)}(\nu_a^n,\mu))$ for $\mu > \mu_a^n$  where $n=N_a^t$ is the number of pulls of arm $a$, $\mu_a^n$ is the mean of the empirical distribution $\nu_a^n$ of the $n$ rewards collected from arm $a$ at time $t$. The exponential dependency on the size of the support $m$ directly translates in a regret bound that is exponential in $m$. On the contrary, by using Theorem~\ref{thm:bound_dbc}, we obtain a sharper deviation inequality of the form $\P(w_a^t f \geq \mu) \gtrsim n^{-1/2} \exp(-n \Kinf^{(m)}(\nu_a^n,\mu))$  which allows us to obtain a regret bound independent of the size of the support. Remark also that our deviation inequality is very close to the one provided by \citet{tiapkin2022dirichlet}. However let us point out to two major differences. First, in order to prove their inequality, the authors in \citet{tiapkin2022dirichlet} need an extra prior Dirichlet parameter $\alpha_{2m}^0$ for an artificial reward of size $2$. Furthermore,  this parameter has to be of order $\cO(\log(T))$ instead of a constant as in our current setting. The latter improvement is essential for deriving the bounds of Theorem~\ref{thm:multinomial_ts_regret}.

\subsection{Bounded reward}

In this section, we assume that $\cV =\cP[0,1]$ is the space of all finite probability measures supported on the segment $[0,1]$. Note that for this setting $\Kinf^{\cV}(\nu_a,\mu) = \Kinf(\nu_a,\mu)$ is defined in Section~\ref{sec:main_results} for $\ub = 1$.
Next, we extend the \MTS algorithm to the bounded reward setting by randomized rounding as discussed by \citet{riou20a}, see also \citet{agrawal2013further}. We fix a size $m$ and consider the finite grid $\{0,1/m,\ldots,1\}$. If  $Y \in [i/m,(i+1)/m)$ for some $i\in\{0,\ldots,m-1\}$, we sample a new reward $\tY = (i+B)/m$ with $B\sim \Ber(mY-i) $ that has the same expectation as $Y$ and follows  multinomial distribution. Then we can just feed the transformed reward $\tY$ to the \MTS algorithm. We call this algorithm rounded multinomial Thompson sampling (\ourAlgBounded) and provide a complete description in Algorithm~\ref{alg:RMTS} of Appendix~\ref{app:mts_algorithms}.
\par
We denote by $\nu_a^{(m)}$ the distribution of the transformed reward $\tY$ from the arm $a$. Thanks to the Lemma~6 by \cite{riou20a}, we can quantify the effect of the rounding procedure on the minimal Kullback-Leibler divergence
\[
    \Kinf(\nu_a, \mustar) -\frac{1}{m(1-\mustar) - 1} \leq \Kinf^{(m)}(\nu_a^{(m)}, \mustar) \leq \Kinf(\nu_a, \mustar)\,.
\]
Thus, if we choose a grid-size of order $m = \cO(\log T)$, by combining the previous inequality and the regret bound for \MTS, we prove that \ourAlgBounded is problem-dependent optimal for the case of bounded rewards. Let us stress  that such a result can be obtained only if one has a regret bound for \MTS that does not depend on the size of the support $m$.

\begin{theorem}
\label{thm:bounded_ts_regret}
    For \ourAlgBounded, with the same prior as in Theorem~\ref{thm:multinomial_ts_regret}, it holds for all sub-optimal arms $a,$
\[
        \E[N^T_a] \leq  \frac{\log T}{\Kinf(\nu_a,\mustar)}  + \cO\left( \log^{9/10}(T) \right)
\]
    and consequently 
    \[
        \regret^T \leq \sum_{a: \Delta_a > 0} \frac{\Delta_a \log(T)}{\Kinf(\nu_a,\mustar)} + \cO\left( \log^{9/10}(T) \right).  
    \]
\end{theorem}

Interestingly \ourAlgBounded only requires a space-complexity of order $\cO(K\log(T))$ and has per-round time-complexity of order $\cO(K\log(T))$. Compared to  the space-complexity of order $\cO(T)$ and time-complexity of order $\cO(T)$ for \NPTS, we observe a significant improvement  while preserving the problem-dependent optimality.

\begin{remark}

The \ourAlgBounded algorithm assumes that the time horizon $T$ is known in advance. However, in cases where $T$ is unknown, we can use a doubling trick with the same time complexity. The modified algorithm works as follows: we keep track of all the rewards obtained from the real model, but after reaching $2^k$ episodes, we increase the size of the support from $k$ to $k+1$ by applying the rounding procedure again to each of the $2^k$ samples. The total cost of the rounding procedure after $T$ episodes will be no greater than $\sum_{k=0}^{\lfloor \log_2 T \rfloor } 2^k \leq 2T$. This means that the amortized per-round time complexity of the rerounding step is just $\cO(1)$. Although this modification increases the space complexity to $\cO(T)$, it still preserves the original time complexity of $\cO(K \log(T))$ in the case of an unknown $T$.

\end{remark}

\section{Proof of the main results}

In this section we present proofs of our main results on deviation bounds for Dirichlet projections: Theorem~\ref{thm:bound_dbc} and Corollary~\ref{cor:bound_dbc_all_u}.

\subsection{Proof of Theorem~\ref{thm:bound_dbc}}
First based on the integral representation for the density of weighted Dirichlet sums (Proposition~\ref{prop:density_linear_dirichlet}) and its nonasymptotic expansion (Proposition~\ref{prop:asymptotics_integral}), we prove the following lemma.
\begin{lemma}\label{lem:bounds_dirichlet_density}
     Consider a function $f \colon \{0,\ldots,m\} \to [0,b]$ such that $f(0) = 0$ and $f(m) = b$, and a fixed positive vector $\alpha = (\alpha_0, \alpha_1, \ldots, \alpha_m) \in \R_+^{m+1}$. Define $\up \in \simplex_m$ such that $\up(\ell) = \alpha_\ell / \ualpha$ for $\ualpha = \sum_{j=0}^m \alpha_j$. Let $\varepsilon > 0$ and assume that $\alphamin = \min\{\alpha_0, \alpha_n\} \geq c_0 \cdot \varepsilon^{-2}$, where
     \begin{equation}\label{eq:c0_definition}
        c_0 = \frac{2}{\pi} \left( 8 + \frac{49 \sqrt{6}}{9} \right)^2.
    \end{equation}
    \begin{itemize}
        \item Let $w^+ \sim \Dir(\alpha^+)$ for $\alpha^+ = (\alpha_0 , \alpha_1,\ldots,\alpha_{m-1},\alpha_m+1)$. Then for any $u \in (\up f, \ub),$ the density of the random variable $Z^+ = w^+ f$ could be lower-bounded as follows
        \[
            p_{Z^+}(u) \geq (1 - \varepsilon) \sqrt{\frac{\ualpha}{2\pi (1- \lambda^\star(\ub - u))^2 \sigma^2}} \cdot \exp(-\ualpha \Kinf(\up, u, f)).
        \]
        \item Let $w^- \sim \Dir(\alpha^-)$  for $\alpha^- = (\alpha_0+1, \alpha_1, \ldots, \alpha_{m-1}, \alpha_m)$. Then for any $u \in (\up f, \ub)$ the density of a random variable $Z^- = w^- f$ could be upper-bounded as follows
        \[
            p_{Z^-}(u) \leq (1 + \varepsilon) \sqrt{\frac{\ualpha}{2\pi (1 + \lambda^\star u)^2 \sigma^2}} \cdot \exp(-\ualpha \Kinf(\up, u, f)).
        \]
    \end{itemize}
    
\end{lemma}
\begin{proof}
    We start the proof from the combination of Proposition~\ref{prop:density_linear_dirichlet} and Proposition~\ref{prop:asymptotics_integral} for $\ell = m$
    \[
        p_{Z^+}(u) \geq \frac{\ualpha}{\sqrt{2\pi}}  \left( 1 - \frac{c_1 \cdot \rme^{-c_{\kappa} \alphamin}}{\sqrt{2 \pi \alphamin}} - \frac{c_2}{\sqrt{2\pi \alphamin}} - \frac{c_3 \cdot \rme^{-c_{\kappa} \alphamin}}{\sqrt{2\pi}} \right) \frac{\exp(-\ualpha \Kinf(\up, u, f))}{(1 - \lambda^\star(\ub - u)) \sqrt{\ualpha \cdot \sigma^2}}.
    \]
    The goal is to find a lower bound for $\alphamin$ such that the following inequality will be guaranteed.
    \begin{align}\label{eq:lb_dirichlet_density_approx}
        \frac{c_1 \cdot \exp(-c_{\kappa} \alphamin)}{\sqrt{2 \pi c_{\kappa} \alphamin}} + \frac{c_2}{\sqrt{2\pi \alphamin}} + \frac{c_3 \cdot \exp(-c_{\kappa} \alphamin)}{\sqrt{2\pi}} \leq \varepsilon.
    \end{align}
    Notice that the second term is the most dominant for large $\alphamin$. First, we have to guarantee that
    \[
        \frac{c_2}{\sqrt{2\pi \alphamin}} \leq \varepsilon/2, 
    \]
    which satisfied for $\alphamin \geq \frac{4 c_2^2}{2\pi \varepsilon^2}$. To satisfy \eqref{eq:lb_dirichlet_density_approx} it is enough to guarantee that
    \[
        \left( \frac{c_1}{2 c_2 \sqrt{c_{\kappa}}} \cdot \varepsilon + \frac{c_3}{\sqrt{2\pi}} \right) \exp(-c_{\kappa} \alphamin) \leq \varepsilon/2 \iff \alphamin \geq \frac{1}{c_{\kappa}} \log\left( \frac{c_1}{c_2 \sqrt{c_{\kappa}}}  + \frac{2c_3}{\sqrt{2\pi}} \cdot \frac{1}{\varepsilon} \right).
    \]
    Next, we notice that for any $\varepsilon \in (0,1)$ the following inequality holds due to the specific choice of $c_1,c_2,c_3,c_{\kappa}$
    \[
        \frac{2c_2^2}{\pi \varepsilon^2} \geq  \frac{1}{c_{\kappa}} \log\left(  \frac{c_1}{c_2 \sqrt{c_{\kappa}}}  + \frac{2c_3}{\sqrt{2\pi}} \cdot \frac{1}{\varepsilon} \right).
    \]
    Therefore we have that for $\alphamin \geq 2c_2^2 / (\pi \varepsilon^2)$ the first statement of Lemma~\ref{lem:bounds_dirichlet_density} holds. For the second statement is is enough to apply the same combination of Proposition~\ref{prop:density_linear_dirichlet} and Proposition~\ref{prop:asymptotics_integral} but for $\ell = 0$.
\end{proof}
Before proceeding with the final proof, we derive one important technical result. 
\begin{lemma}\label{lem:kinf_bounds}
    For any $u \in (\up f, \ub)$ it holds 
    \[ 
    \frac{1}{2} (\lambda^\star)^2 \sigma^2  \big(1-\lambda^\star(\ub-u)\big)^2 \leq \Kinf(\up, u, f) \leq \frac{1}{2} (\lambda^\star)^2 \sigma^2  \big(1+\lambda^\star u\big)^2\,. 
    \]
\end{lemma}

\begin{proof}
Define the function $\phi_u(\lambda) = \E[\log\big(1-\lambda(f(X)-u)\big)]$ and $\lambda_u = \lambda^\star$. Remark that $\sigma^2 = -\phi_u''(\lambda_u)$. Thanks to the Taylor expansion of $\phi_u$ and the definition of $\lambda_u$ it holds 
\begin{align*}
    0 = \phi_u(0) &= \phi_u(\lambda_u) + 0 + \frac{\lambda_u^2}{2} \phi_u''(y\lambda_u) \iff \phi_u(\lambda_u) = \frac{\lambda_u^2}{2} (-\phi_u''(y\lambda_u))
\end{align*}
for some $y \in (0,1)$. 
We will bound the negative second derivative that appears above from both sides. First, note that 
\[
-\phi_u''(y\lambda_u) = \E\left[ \frac{(f(X)-u)^2}{\big(1-\lambda_u(f(X)-u)\big)^2}\left(\frac{1-\lambda_u(f(X)-u)}{1-y\lambda_u(f(X)-u)}\right)^2\right]\,.
\]
Let us consider two cases.

1) $f(X) \leq u$. Using the fact that $y\in (0,1)$ we can derive the following bounds
\[
    1 + \lambda_u u \geq \frac{1-\lambda_u(f(X)-u)}{1-y\lambda_u(f(X)-u)} \geq 1 \geq 1 - \lambda_u (\ub - u)\,.
\]

2) $f(X) \geq u$. By the symmetric argument 
\[
    1 + \lambda_u u \geq 1 \geq \frac{1-\lambda_u(f(X)-u)}{1-y\lambda_u(f(X)-u)} \geq 1-\lambda_u(b-u)\,.
\]
In particular, using the definition of $\phi''(\lambda_u) = -\sigma^2$, it entails that 
\[
    \sigma^2 \big(1-\lambda_u(b-u)\big)^2 \leq -\phi_u''(y\lambda_u) \leq \sigma^2 \big(1+\lambda_u u\big)^2\,.
\]
Plugging this inequality in the integral representation of $\phi_u$ allows us to conclude the statement.
\end{proof}

Using this lemma we may proceed with the proof of our final results.
\begin{proof}[Proof of Theorem~\ref{thm:bound_dbc}]
    Define $Z^{+} = wf$ for $w \sim \Dir(\alpha^+)$. By Lemma~\ref{lem:bounds_dirichlet_density},
   \begin{align*}
        \P_{w \sim \Dir(\alpha^+)}\left[wf \geq \mu\right] &= \int_{\mu}^{\ub} p_{Z^+}(u) \rmd u \geq (1 - \varepsilon) \sqrt{\frac{\ualpha}{2\pi}}  \cdot  \int_{\mu}^{\ub}  \frac{ \exp(-\ualpha \Kinf(\up, u, f))}{\sqrt{\sigma^2 (1 - \lambda^\star(\ub - u))^2}} \, \rmd u.
    \end{align*}
    By Theorem 6 by \cite{honda2010asymptotically},
    \[
        \frac{\partial}{\partial u}\Kinf(\up, u, f) = \lambda^\star.
    \]
    Thus, we can define a change of variables $t^2/2 = \Kinf(\up, u, f), t \rmd t = \lambda^\star \rmd u$ and write
    \begin{align*}
        \P\left[Z^+ \geq \mu\right] \geq (1 - \varepsilon)\int_{\sqrt{2 \Kinf(\up, \mu, f)}}^{+\infty} D^+(u) \sqrt{\frac{\ualpha}{2\pi}} \exp(-\ualpha t^2/2) \rmd t,
    \end{align*}
    where $D^+(u)$ is defined as follows
    \[
        D^+(u) = \sqrt{\frac{2 \Kinf(\up, u, f)}{(\lambda^\star)^2 \sigma^2 (1 - \lambda^\star(\ub - u))^2}}.
    \]
    By Lemma~\ref{lem:kinf_bounds}, $D^+(u) \geq 1$ and hence for $g \sim \cN(0,1)$
    \begin{align*}
        \P\left[Z^+ \geq \mu\right] &\geq (1 - \varepsilon)\int_{\sqrt{2 \Kinf(\up, \mu, f)}}^{+\infty} \sqrt{\frac{\ualpha}{2\pi}} \rme^{-\ualpha t^2/2} \rmd t = (1 - \varepsilon) \P\big( g \geq \sqrt{2 \ualpha \Kinf(\up, \mu, f)} \big).
    \end{align*}
For the second part of the statement let us define $Z^{-} = wf$ for $w \sim \Dir(\alpha^-)$. By Lemma~\ref{lem:bounds_dirichlet_density},
   \begin{align*}
        \P_{w \sim \Dir(\alpha^-)}\left[wf \geq \mu\right] &= \int_{\mu}^{\ub} p_{Z^-}(u) \rmd u \leq (1 + \varepsilon) \sqrt{\frac{\ualpha}{2\pi}}  \cdot  \int_{\mu}^{\ub}  \frac{ \exp(-\ualpha \Kinf(\up, u, f))}{\sqrt{\sigma^2 (1 + \lambda^\star u)^2}} \, \rmd u.
    \end{align*}
    By a change of variables $t^2/2 = \Kinf(\up, u, f), t \rmd t = \lambda^\star \rmd u$ 
    \begin{align*}
        \P\left[Z^{-} \geq \mu\right] \leq (1 + \varepsilon)\int_{\sqrt{2 \Kinf(\up, \mu, f)}}^{+\infty} D^-(u) \sqrt{\frac{\ualpha}{2\pi}} \exp(-\ualpha t^2/2) \rmd t,
    \end{align*}
    where $D^-(u)$ is defined as follows
    \[
        D^-(u) = \sqrt{ \frac{2 \Kinf(\up, u, f)}{(\lambda^\star)^2 \sigma^2 (1 + \lambda^\star u)^2} }.
    \]
    By Lemma~\ref{lem:kinf_bounds}, $D^-(u) \leq 1$ and hence
    \begin{align*}
        \P\left[Z^- \geq \mu\right] &\leq (1 + \varepsilon) \P_{g \sim \cN(0,1)}\big[ g \geq \sqrt{2 \ualpha \Kinf(\up, \mu, f)} \big].
    \end{align*}
\end{proof}

\subsection{Proof of Corollary~\ref{cor:bound_dbc_all_u}}

\begin{proof}
    By Theorem~\ref{thm:bound_dbc} both statements holds for $\mu \geq \up f$. Thus, we may assume that $\mu < \up f$ and therefore
    \[
        A(\up, \mu, f) = - \sqrt{2 \Kinf(\up, b-\mu, b-f)}.
    \]
    Let us start from the lower bound. We have for $w \sim \Dir(\alpha^+)$
    \begin{equation}\label{eq:reduction_to_high_expectation}
        \P(wf \geq \mu) = 1 - \P(wf \leq \mu) = 1 - \P(w(b-f) \geq b-\mu).
    \end{equation}
    Notice that for a function $b-f$ the role of $0$ and $m$ atoms has changed: $(b-f)(0) = b$ and $(b-f)(m) = 0$. To obtain the upper bound on the probability in the right-hand side, we can simply apply upper bound from Theorem~\ref{thm:bound_dbc}. 

    For any vector $x \in \R^{m+1}_+$ let us define the reversing operation $\rev(x)$ such that $y = \rev(x) \iff y_i = x_{m-i}$ for any $i \in \{0,\ldots,m\}$.  Then we define a vector $\beta = \rev(\alpha)$ and for this vector we may define $\beta^-$ and $\beta^+$ by increasing values of $\beta_0$ and $\beta_m$ by 1 correspondingly. Additionally, let us define a measure $\bar{q} \in \simplex_m$ such that $\bar{q} = \rev(\up)$ and a function $g = \rev(f)$. For $w \sim \Dir(\alpha^+)$ we have $w' = \rev(w)$ has distribution $\Dir(\beta^-)$ and, moreover, $w' (b - g) = w(b-f)$ and  $\bar{q}(b-g) \leq b-\mu$, therefore Theorem~\ref{thm:bound_dbc} applicable
    \[
        \P_{w' \sim \Dir(\beta^-)}[w' (b - g) \geq b-\mu] \leq (1 + \varepsilon)\P\left[ \zeta \geq \sqrt{2 \ualpha \Kinf(\bar{q}, b-g, b-\mu)} \right]
    \]
    for $\zeta \sim \cN(0,1)$. By noticing that $\Kinf(\bar{q}, b-g, b-\mu) = \Kinf(\up, b-f, b-\mu)$ we obtain the following upper bound
    \[
         \P(w(b-f) \geq b-\mu) \leq (1 + \varepsilon) \P\left[ \zeta \geq - \sqrt{\ualpha} \cdot A(\up, \mu, f) \right].
    \]
    Therefore, we have since $A(\up, \mu, f) < 0$ and $\zeta$ is symmetric
    \[
        \P(wf \geq \mu) \geq 1 - (1 + \varepsilon)  \P\left[ \zeta \geq - \sqrt{\ualpha} \cdot A(\up, \mu, f) \right] \geq (1-\varepsilon)  \P\left[ \zeta \geq \sqrt{\ualpha} \cdot A(\up, \mu, f) \right].
    \]
    The same holds for the upper bound, starting from representation \eqref{eq:reduction_to_high_expectation} and using a lower bound on probability of $\P[w(\ub - f) \geq b-\mu]$ for a reversed model.
\end{proof}

\appendix
\newpage
\part{Appendix}
\parttoc
\newpage

\section{Proof of Lemma~\ref{lem:bounds_dirichlet_density}}

Throughout this section we will often use the following notations.  Let $\Fclass_m(b) = \{f \colon \{0,\ldots,m\} \to [0,b] \mid f(0) = 0, f(m) = \ub \}$. For any $\alpha = (\alpha_0, \alpha_1, \ldots, \alpha_m) \in \R^{m+1}_{++}$ define  $\up = \up(\alpha) \in \simplex_{m}$ such that $\up(\ell) = \alpha_\ell/\ualpha, \ell = 0, \ldots, m$, where $\ualpha = \sum_{j=0}^m \alpha_j$.

\subsection{Density of weighted sum of the Dirichlet distribution}

In this section we compute the density of a random variable $Z = w f$, where $w \sim \Dir(\alpha)$ and $f \in \Fclass_m(b)$. 

\begin{proposition}\label{prop:density_linear_dirichlet}
    Let $f \in \Fclass_m(b)$ and $\alpha = (\alpha_0, \alpha_1, \ldots, \alpha_m) \in \R_+^{m+1}$ such that $\ualpha = \sum_{j=0}^m \alpha_j > 1$. Assume that $Z$ is not degenerate. Then for any $0 \leq u < \ub$
    \[
        p_{Z}(u) = \frac{\ualpha - 1}{2\pi} \int_\R \prod_{j=0}^m \left( 1 + \rmi(f(j) - u)s \right)^{-\alpha_j} \rmd s.
    \]
\end{proposition}
Proof of Proposition~\ref{prop:density_linear_dirichlet} will be given at the end of this paragraph.

A function $g \colon \R^{m+1} \to \R$ is called a positive homogeneous on a cone $A \subseteq \R^{m+1}$ of degree $t$ if for any $\gamma > 0$ and $x \in A$ we have $g(\gamma x) = \gamma^t g(x)$. Define $\widetilde{\simplex}_m = \conv(0, \simplex_m) = \{ w \in \R_+^{m+1} : \sum_{j=0}^m w_j \leq 1\}$ as a pyramid with a base $\simplex_m$ and apex at $0$. For $r > 0$ we write $\simplex_m(r) = \{w \in \R^{m+1}_{+}: \sum_{\ell=0}^m w_\ell = r\}$. Then, clearly $\simplex_m = \simplex_m(1)$. For any $a \in \R^{m+1}$ define $\Hset_a = \{ w \in \R^{m+1} : \langle a, w \rangle = 0\}$. Also, for any matrix $M \in \R^{m+1 \times k}$ define $\Hset_M = \{w \in \R^{m+1} : Mw = 0 \}$.

For any measurable set $A \subseteq \R^{m+1}$ of dimension $n < m+1$ and any function $g \colon \R^{m+1} \to \R$ define 
\[
    I_n(g, A) = \int_{A} g(w) \cH^n(\rmd w),
\]
where $\cH^n$ is an $n$-dimensional Hausdorff measure (see \citealp{evans2018measure}, for definition). If $A = \cL(Y)$ for a linear map $\cL \colon \R^n \to \R^{m+1}$ and $Y \subseteq \R^n$, then we can write
\[
    I_n(g, \cL(Y)) = [\cL] \cdot \int_{Y} g(\cL(y)) \lambda_n(\rmd y),
\]
where $\lambda_n$ is an $n$-dimensional Lebesgue measure on $Y$ and $[\cL]$ is a Jacobian of the map $\cL$ that could be computed as $[\cL] = \sqrt{\det(\cL \cL^\top)}$. Let us define an affine map $\cL^t_a \colon \R^m \to \R^{m+1}$ that transforms $\R^m$ to $\Hset_a^t$ by mapping $x_1,\ldots,x_m$  to $w_1,\ldots,w_m$ and $w_0 = \frac{t - \sum_{j=1}^m a_j x_j}{a_0}$ for $a_0 > 0$ (without loss of generality). The linear part of this map has a Jacobian that is equal to $[\cL^t_a] = \frac{\norm{a}_2}{a_0}$ (see Lemma~\ref{lem:jacob}). Additionally, define $\cL_a = \cL_a^0$. Define $\cL_{M}$ as a maps onto $\cH_M$ for a matrix $M$, for a definition see Appendix~\ref{app:linear_algebra}. For two vectors $a,b \in \R^{m+1}$ we define $\cH_{[a,b]}^{t_1,t_2} = \{ x \in \R^{m+1} : \langle a, x\rangle = t_1, \langle b,x \rangle = t_2 \}$ and a corresponding canonical map onto this set as $\cL_{[a,b]}^{t_1,t_2}$. Notice that this map depends only on the space $\mathrm{span}(a,b)$ and not vectors itself.

\begin{lemma}\label{lem:pyramid_volume}
    Let $g$ be a positively homogeneous function of degree $t > -m$ on $\R_{++}^{m+1}$ and assume that $a_0 \not = a_1$. Then we have
    \[
        I_{m}(g, \widetilde{\simplex}_{m} \cap \Hset_a) = \frac{\norm{a}_2}{ \vert a_0 - a_1 \vert \cdot [\cL_{[a,\bOne]}]} \frac{ I_{m-1}(g, \simplex_m \cap \Hset_a)}{m+t}.
    \]
\end{lemma}
\begin{proof}
    First, we apply the change of variables formula \citep[3.3.3]{evans2018measure} by using a map $\cL_a$
    \[
        I_m(g, \widetilde{\simplex}_m \cap \Hset_a) = [\cL_a] \int_{\cL_a(x) \geq 0,\ \bOne^\top \cL_a(x) \leq 1} g(\cL_a(x)) \lambda^m(\rmd x).
    \]
    We can define a map that acts as a scalar product with $\cL_a^\top \bOne$ and apply the coarea formula \citep[3.4.3]{evans2018measure}
    \[
        I_m(g, \widetilde{\simplex}_m \cap \Hset_a) = \frac{[\cL_a]}{[\cL_a^\top \bOne]}  \int_{0}^1 \rmd t  \int_{\cL_a(x) \geq 0,\ \bOne^\top \cL_a(x) = t} g(\cL_a(x)) \cH^{m-1}(\rmd x).
    \]
    Next we have to compute the inner integral. To do it, we define a map $\cL^t_{\cL_a^\top \bOne}$ that maps $\R^m$ to a set $\{ x: \bOne^\top \cL_a(x) = t \}$. Notice that the Jacobian of this map does not depend on the parameter $t$. Therefore by the change of variables formula we have
    \begin{align*}
        \int_{\cL_a(x) \geq 0,\ \bOne^\top \cL_a(x) = t} g(\cL_a(x)) \cH^{m-1}(\rmd x) &= [\cL_{\cL_a^\top \bOne}] \int_{\cL_{[a, \bOne]}^{0,t}(y) \geq 0} g(\cL_{[a, \bOne]}^{0,t}(y)) \lambda^{m-1}(\rmd y) \\
        &= \frac{[\cL_{\cL_a^\top \bOne}]} {[\cL_{[a,\bOne]}]} I_{m-1}(g, \simplex_m(t) \cap \Hset_a).
    \end{align*}
    where we used Lemma~\ref{lem:maps_product} to simplify the map $\cL_a \circ \cL^t_{\cL_a \bOne} = \cL_{[a,\bOne]}^{0,t}$. Using definition of a positive homogeneous function and properties of the Hausdorff measure $\cH^{m-1},$ we derive
    \begin{align*}
        I_m(g, \widetilde{\simplex}_m \cap \Hset_a) &= \frac{[\cL_a][\cL_{\cL_a^\top \bOne}] }{[\cL_{[a,\bOne]}] [\cL_a^\top \bOne] } \int_0^1 I_{m-1}(g, \simplex_m(t) \cap \Hset_a) \rmd t \\
        &= \frac{[\cL_a][\cL_{\cL_a^\top \bOne}]}{[\cL_{[a,\bOne]}] [\cL_a^\top \bOne]} \frac{ I_{m-1}(g, \simplex_m \cap \Hset_a)}{m+t}.
    \end{align*}
    Next we simplify the expression by exact formulas for Jacobian of maps to hyperplanes (Lemma~\ref{lem:jacob})
    \[
        \frac{[\cL_a][\cL_{\cL_a^\top \bOne}]}{[\cL_a^\top \bOne]} = \frac{\norm{a}_2 \norm{\cL_a^\top \bOne}_2}{\vert a_0 \vert  \cdot \vert (\cL_a^\top \bOne)_{0} \vert \cdot \norm{\cL_a^\top \bOne}_2} = \frac{\norm{a}_2}{ \vert (\cL_a^\top \bOne)_0 \vert  \cdot \vert a_0 \vert }
    \]
    where without loss of generality we assume that $(\cL_a^\top \bOne)_{0} \not = 0$ and $a_0 \not = 0$. To finally simplify the expression, it is enough to notice that $(\cL_a^\top \bOne)_0 = (a_0 - a_1)/a_0$.
\end{proof}

Next we provide another representation of the integral $I_{m}(g, \widetilde{\simplex}_m\cap \Hset_a)$. We follow \citet{lasserre2020simple} and use the same technique based on the Laplace transform.
\begin{lemma}\label{lem:simplex_laplace}
    Let $g$ be a positively homogeneous function of degree $t$ on $\R_{++}^{m+1}$ such that $t > -(1+m)$ and $\int_{\widetilde{\simplex}_m} \vert g(w) \vert \rmd w < \infty$. Then
    \[
        I_{m}(g, \widetilde{\simplex}_m \cap \Hset_a) = \frac{1}{\Gamma(1 + m + t)} \int_{\R_+^{m+1} \cap \Hset_a} g(w) \exp\left(-\sum_{\ell=0}^m w_\ell\right) \cH^m(\rmd w).
    \]
\end{lemma}
\begin{proof}
    Consider $h(y) = \int_{w \geq 0, \langle \bOne, w \rangle \leq y, \langle a, w \rangle =0} g(w) \cH^m(\rmd w)$. Clearly, $h(1) = I_{m}(g, \widetilde{\simplex}_m \cap \Hset_a)$. Since $g$ is positively homogeneous function we get $h(y) = y^{m+t} h(1)$. This implies that the Laplace transform of $h$ is equal to $L(\lambda) = \int_0^\infty h(y) \rme^{-\lambda y} \rmd y = h(1) \frac{\Gamma(m+t+1)}{\lambda^{m + t + 1}}$. On the other hide, the Laplace transform $L(\lambda)$ may be calculated via a linear parametrization of the subspace $\langle a,w \rangle = 0$ using the map $\cL_a$ and the Fubini theorem
    \begin{align*}
        L(\lambda) &= \int_0^\infty \rme^{-\lambda y} \left[ \int_{w \in \R^{m+1}_+, \langle \bOne,w \rangle \leq y, \langle a,w \rangle =0 }  g(w) \cH^m(\rmd w) \right] \rmd y\\
        &= [\cL_a]\int_{\cL_a(x) \geq 0} \rmd x \cdot g(\cL_a(x)) \cdot \left[ \int_{\langle \bOne, \cL_a(x) \rangle \le y}  \rme^{-\lambda y} \rmd y \right] \\
        &= \frac{[\cL_a]}{\lambda} \int_{\cL_a(x) \geq 0} g(\cL_a(x)) \exp\left( -\lambda \langle \bOne, \cL_a(x) \rangle \right) \rmd x \\
        &= \frac{[\cL_a]}{\lambda^{m+t+1}} \int_{\cL_a(x) \geq 0} g(\cL_a(x)) \exp\left( -\langle \bOne, \cL_a(x) \rangle \right) \rmd x \\
        &= \frac{1}{\lambda^{m+t+1}} \int_{\R^{m+1}_+ \cap \Hset_a} g(w) \exp\left( -\sum_{\ell=0}^m w_\ell  \right) \cH^m(\rmd w).
    \end{align*}
    Identifying two ways of computation of the Laplace transform, we finish the proof.
\end{proof}

We now compute the integral in the r.h.s.\,of Lemma~\ref{lem:simplex_laplace}. We shall use the Fourier transform method and follow the approach of \citet{dirksen2015sections}. 
\begin{lemma}\label{lm:simplex_fourier_integral}
    Let \(g(w) = w_0^{\alpha_0 -1} \cdot \ldots \cdot w_{m}^{\alpha_m - 1}\). Then we have
    \[
        \int_{\R_+^{m+1} \cap \Hset_a} g(w) \exp\left(-\sum_{i=0}^m w_i\right) \cH^m(\rmd w) = \frac{\norm{a}_2  \cdot \prod_{j=0}^m \Gamma(\alpha_j)}{2\pi} \int_{\R} \prod_{j=0}^m (1 + \rmi a_j \tau)^{-\alpha_j} \rmd \tau.
    \]
\end{lemma}
\begin{proof}
    Denote for any $t \in \R$
    \[
        G(t) = \int_{\forall j: \langle a, w\rangle = t, w_j \geq 0} g(w) \exp\left(-\sum_{\ell=0}^m w_\ell\right) \cH^m (\rmd w).
    \]
    Next we write down the Fourier transform of $G$ for $\tau \in \R$
    \begin{align*}
        \mathcal{F}[G](\tau) &= \frac{1}{\sqrt{2\pi}} \int_{\R} \rme^{-\rmi t \tau  } \rmd t \int_{\langle a,w\rangle=t, w \geq 0} g(w) \exp\left(-\sum_{\ell=0}^m w_\ell\right) \cH^{m} (\rmd w).
    \end{align*}
    Let us move $\rme^{-\rmi t \tau}$ under the sign of the second integral and replace $t$ with $\langle a, w\rangle$. Notice that it is correct since the functions $g(w) \exp\left(-\sum_{\ell=0}^m w_\ell\right) \rme^{-\rmi t \tau}$ and $g(w) \exp\left(-\sum_{\ell=0}^m w_\ell\right) \rme^{-i \tau  \langle a, w\rangle}$ are almost surely equal to each other on the set $\langle a, w\rangle =t$. Thus, we have
    \[
        \mathcal{F}[G](\tau) = \frac{1}{\sqrt{2\pi}} \int_{\R} \rmd t \left[ \int_{\langle a, w\rangle =t} g(w) \ind\{w \geq 0\} \exp\left(-\langle w, \bOne + \rmi a \tau \rangle\right) \cH^{m} (\rmd w) \right]
    \]
    Now we may apply a coarea formula for the map defined by a scalar product with $a$. We have
    \[
        \mathcal{F}[G](\tau) = \frac{[a]}{\sqrt{2\pi}} \int_{\R^{m+1}} g(w) \ind\{w \geq 0\} \exp\left(-\langle w, \bOne + \rmi a \tau \rangle\right) \lambda^{m+1} (\rmd w).
    \]
    We arrive at the product of $m$ characteristic functions of independent $\Gamma(\alpha_\ell, 1)$-distributed random variables
    \[
         \mathcal{F}[G](\tau) = \frac{[a]}{\sqrt{2\pi}} \cdot \prod_{j=0}^m \Gamma(\alpha_j) \cdot \frac{1}{(1 + \rmi a_j \tau )^{\alpha_j}},
    \]
    where $A^{(j)}$ is $j$-th column of matrix $A$. Finally, by the inverse Fourier transform and noticing that $[a] = \norm{a}_2$
    \[
        G(0) = \frac{\norm{a}_2}{\sqrt{2\pi}} \int_{\R}   \mathcal{F}[G](\tau)  \rmd\tau = \frac{\norm{a}_2 \cdot \prod_{j=0}^m \Gamma(\alpha_j)}{2\pi} \int_{\R} \prod_{j=0}^m (1 + \rmi a_j \tau )^{-\alpha_j} \rmd\tau.
    \]
\end{proof}
\begin{remark}
    Notice that the value of the integral is the right-hand side is real because the function under integral has even real part and odd imaginary one.
\end{remark}

\begin{corollary}\label{cor:simplex_integral_subspace_zero}
   Let $g(w) = \Gamma(\ualpha) \prod_{j=0}^m w_j^{\alpha_j - 1}/\Gamma(\alpha_j)$, where $\ualpha > 1$ and $a \in \R^{m+1}$ such that $a_0 \not = a_1$. Then
    \[
        I_{m-1}(g, \simplex_m \cap \Hset_a) = (\ualpha - 1) \cdot [\cL_{[a,\bOne]}] \cdot \vert a_0 - a_1 \vert \cdot \frac{1}{2\pi}\int_{\R} \prod_{j=0}^m (1 + \rmi a_j s)^{-\alpha_j} \rmd s.
    \]
\end{corollary}
\begin{proof}
    Notice that $g$ is positively homogeneous function of degree $\ualpha - (m+1) > -m$ on $\R^{m+1}_{++}$. Hence,  we may apply Lemma \ref{lem:pyramid_volume} and Lemma \ref{lem:simplex_laplace}. We obtain
    \begin{align*}
         I_{m-1}(g, \simplex_m \cap \Hset_a) &= \frac{(\ualpha - 1) \cdot [\cL_{[a,\bOne]}] \cdot \vert a_0 - a_1 \vert}{\norm{a}_2} \cdot I_{m}(g, \widetilde{\simplex}_m \cap \Hset_a) \\
         &= \frac{(\ualpha - 1) \cdot [\cL_{[a,\bOne]}] \cdot \vert a_0 - a_1 \vert }{\norm{a}_2 \cdot \Gamma(\ualpha)} \int_{\R_+^{m+1} \cap \Hset_a} g(w) \exp\left(-\sum_{\ell=0}^m w_\ell\right) \cH^m(\rmd w).
    \end{align*}
    The last integral could be computed by Lemma~\ref{lm:simplex_fourier_integral}. We have
    $$
        I_{m-1}(g, \simplex_m \cap \Hset_a) = (\ualpha - 1) \cdot [\cL_{[a,\bOne]}] \cdot \vert a_0 - a_1 \vert \cdot \frac{1}{2\pi}\int_{\R} \prod_{j=0}^m (1 + \rmi a_j s)^{-\alpha_j} \rmd s.
    $$
\end{proof}

Now we are ready to prove Proposition~\ref{prop:density_linear_dirichlet}.
\begin{proof}[Proof of Proposition~\ref{prop:density_linear_dirichlet}.]
    First, we give a formula for $p_Z$ in terms of $I_{m-1}$. We start from rewriting the probability in terms of a usual integral
    \[
        \P_{w \sim \Dir(\alpha)}[wf \leq \mu] = \int_{w \geq 0, \sum_{i=1}^m w_i \leq 1, wf \leq \mu} g\left(1-\sum_{i=1}^m w_i, w_1,\ldots,w_m\right) \rmd w_1,\ldots, \rmd w_m
     \]
     where $g(w) = \Gamma(\ualpha) \prod_{j=0}^m w_j^{\alpha_j - 1}/\Gamma(\alpha_j)$ is the density of the Dirichlet distribution. We note that this transform exactly defines a map $\cL_\bOne^1$. Then we apply changing of variables formula  \citep[3.4.3]{evans2018measure} using as a map a scalar product with a vector $c = (\cL_{\bOne}^1)^\top f$
    \begin{align*}
        \P_{w \sim \Dir(\alpha)}[wf \leq \mu] = \frac{1}{[c]} \int_0^\mu \left[ \int_{\cL_{\bOne}^1(x) \geq 0, f^\top \cL_{\bOne}^1(w) = u} g(\cL_{\bOne}^1(x)) \cH^m(\rmd x) \right] \rmd u
    \end{align*}
    Then we apply changing of variables formula \citep[3.3.3]{evans2018measure} to the inner integral using parametrization through map $\cL^u_{c}$
    \begin{align*}
        \P_{w \sim \Dir(\alpha)}[wf \leq \mu]  &= \frac{1}{[c]} \int_0^\mu \left[ [\cL_{c}^u] \int_{\cL_{\bOne}^1(\cL_{c}^u w) \geq 0} g(\cL_{\bOne}^1(\cL_{c}^u(z))) \rmd z\right] \rmd u
    \end{align*}
    We note that a Jacobian of $\cL_c^u$ does not depend on the shift parameter $u$, therefore it could be moved from the integral sign. Next we apply changing of variables formula for a map $\cL_{\bOne}^1 \circ \cL_c^u$
    
    \begin{align*}
        \P_{w \sim \Dir(\alpha)}[wf \leq \mu]   &= \frac{[\cL_{c}]}{[c][\cL_{\bOne} \circ \cL_c]} \int_0^\mu \left[ \int_{\simplex_m, wf = u} g(w) \cH^{m-1}(\rmd w)\right] \rmd u.
    \end{align*}
    By the expression of $c = \cL_1^\top f$ we can apply Lemma~\ref{lem:maps_product}. As a result, $p_Z$ can be represented as  the following integral
    \begin{align*}
        p_Z(u) &= \frac{[\cL_{\cL_\bOne^\top f}]}{[\cL_\bOne^\top f][\cL_{[\bOne, f]}]} I_{m-1}(g, \simplex_m \cap \Hset^u_{f}),
    \end{align*}
    where  $\Hset^u_f = \{ w \in \R^{m+1} : wf  = u \}$. Unfortunately, we cannot apply the previous result directly because the hyperplane $\Hset^u_f$ does not intersect $0$ in general. To overcome this issue, define the following vector $a(u)_j = f(j) - u$.
    Note that $\langle w, a(u) \rangle = 0$ iff $\langle w, f - u \bOne \rangle = wf - u = 0$, where we used $\langle w ,\bOne\rangle = 1$. Hence $\Hset^u_f \cap \simplex_m = \Hset_{a(u)} \cap \simplex_m$. We can apply Corollary~\ref{cor:simplex_integral_subspace_zero} to the subspace $\Hset_{a(u)}$
    \begin{align*}
        I_{m-1}(g, \simplex_m \cap \Hset^u_{f}) &= \frac{\ualpha - 1}{2\pi} [\cL_{[a(u),\bOne]}] \cdot \vert a_0 - a_1 \vert \cdot \int_{\R} \prod_{j=0}^m (1 + \rmi a(u)_j s)^{-\alpha_j} \rmd s \\
        &=\frac{\ualpha - 1}{2\pi} [\cL_{[f,\bOne]}] \cdot \vert a_0 - a_1 \vert \cdot \int_{\R} \prod_{j=0}^m (1 + \rmi (f(j) - u)s)^{-\alpha_j} \rmd s,
    \end{align*}
    where we used that the matrices $[a(u), \bOne]$ and $[f,\bOne]$ are equivalent by an elementary transformations. Overall, we have the following expression for the density
    \[
        p_Z(u) = \frac{[\cL_{\cL_\bOne^\top f}] \vert f(0) - f(1) \vert }{[\cL_\bOne^\top f] }  \cdot \frac{\ualpha - 1}{2\pi} \int_{\R} \prod_{j=0}^m (1 + \rmi (f(j) - u)s)^{-\alpha_j} \rmd s.
    \]
    Finally, by Lemma~\ref{lem:jacob} we have
    \[
        p_Z(u) = \frac{\vert f(0) - f(1) \vert}{\vert (\cL^\top_{\bOne} f)_0\vert}  \cdot \frac{\ualpha - 1}{2\pi} \int_{\R} \prod_{j=0}^m (1 + \rmi (f(j) - u)s)^{-\alpha_j} \rmd s.
    \]
    By a direct computation we have $\vert (\cL^\top_{\bOne} f)_0\vert = \vert f(0) - f(1) \vert$ and we conclude the statement.
\end{proof}

\subsection{Method of Saddle Point}
\label{sec:saddle-point}
We start from the asymptotic decomposition of the integral that appears in the expression for density of weighted sum of Dirichlet distribution.

\begin{proposition}\label{prop:asymptotics_integral}
     Let $f \in \Fclass_m(\ub)$ and let  $\alpha = (\alpha_0, \alpha_1, \ldots, \alpha_m) \in \R_{>0}^{m+1}$ be a fixed positive vector with $\alphamin \triangleq \min\{\alpha_0, \alpha_m\} \geq 2$, $\ualpha = \sum_{i=0}^m \alpha_i$. Then for any $u \in (\up f, \ub)$ and any fixed $\ell \in \{0,\ldots,m\}$
    \begin{align*}
        \int_\R \frac{\prod_{j=0}^m \left(1 + \rmi(f(j) - u)s \right)^{-\alpha_j}}{\left(1 + \rmi(f(\ell) - u)s \right)} \rmd s &= \sqrt{\frac{2\pi}{\ualpha \, (1 - \lambda^\star(f(\ell) - u))^2\, \sigma^2 }} \cdot \exp\left(-\ualpha \,\Kinf(\up, u, f)\right) \\
        & + \left( R_2( \alpha) - R_1(\alpha) \right) \exp\left(-\ualpha \,\Kinf(\up, u, f)\right) + R_3(\alpha),
    \end{align*}
    where   
    \begin{align*}
        \sigma^2 &= \E_{X \sim \up}\left [\left(\frac{f(X) - u}{1 - \lambda^\star(f(X) - u)}\right)^2\right], 
                \\
        \vert R_1(\alpha) \vert &\leq \frac{c_1}{(1 - \lambda^\star(f(\ell) - u))\sqrt{\ualpha\,\sigma^2}} \cdot \frac{\exp(-c_\kappa \alphamin) }{(c_\kappa \alphamin)^{1/2}}, \\
        \vert R_2(\alpha) \vert &\leq   \frac{c_2}{(1 - \lambda^\star(f(\ell )- u))\sqrt{\ualpha\,\sigma^2}} \cdot \frac{1}{\sqrt{\alphamin}}, \\
        \vert R_3(\alpha) \vert &\leq  \frac{\exp(-\ualpha \Kinf(\up, u, f)) }{1-\lambda^\star(f(\ell) - u)} \cdot \frac{c_3}{\sqrt{\ualpha\, \sigma^2}} \exp(-c_\kappa \alphamin).
    \end{align*}
with $c_1 = 2\sqrt{2}, c_2 =\left(8 + \frac{49 \sqrt{6}}{9}\right), c_3 = \sqrt{5\rme \pi}, c_\kappa = 1/2 \cdot \log\left(5/4 \right)$ and $\lambda^\star$ being a solution to the optimization problem
    $$
        \lambda^\star(\up, u, f) = \argmax_{\lambda \in [0, 1/(\ub-u)]} \E_{X \sim \up}\left[\log(1 - \lambda (f(X) - u)) \right].
    $$
\end{proposition}
\begin{proof}
        Let us first rewrite the integral in the form
    \begin{align}
        I &= \int_\R  \frac{\prod_{j=0}^m \left(1 + \rmi (f(j) - u)s \right)^{-\alpha_j}}{1+\rmi(f(\ell) - u)s}\, \rmd s \notag \\
        &= \int_\R \frac{\exp\left( - \ualpha \sum_{j=0}^m \up_j \log(1 + \rmi (f(j) - u)s) \right)}{1+\rmi(f(\ell) - u)s}\, \rmd s  \notag \\
        &= \int_\R \frac{\exp\left( - \ualpha \, \E_{X \sim \up}[\log(1 + \rmi (f(X) - u)s)] \right)}{1+\rmi(f(\ell) - u)s} \, \rmd s, \label{eq:integral_representation}
    \end{align}
    where we choose the principle branch of the complex logarithmic function. Denote $S(z) = \E_{X \sim \up}[\log(1 + \rmi (f(X) - u)z)]$. In the sequel we shall write for simplicity $\E$ instead of $\E_{X \sim \up}$.
    
    Since $f(X) \leq \ub$, this function is holomorphic for $\vert \im z \vert < 1 / (\ub - u)$ and $\re z \in \R$. The last integral representation \eqref{eq:integral_representation} allows to use the method of saddle point \cite{olver1997asymptotics}.
  Next, we are going to compute the saddle points of the function $S$. To do it, compute the derivative of the function $S$ at complex point $z = x + \rmi y$
    \[
        S'(z) = \rmi \EE{\frac{f(X)-u}{1 + \rmi (f(X) - u)z} } = \EE{\frac{x (f(X) - u)^2 + \rmi (f(X)-u)(1 - y (f(X) - u))}{(1 - y(f(X) - u))^2 + x^2 (f(X) - u)^2} } = 0.
    \]
    Notice that the real part of the expression above is zero if and only if $x = 0$. Therefore the saddle points could be only on the imaginary line $\rmi\R$. They can be found from the equation
    \[
        S'(\rmi y) = \rmi\EE{\frac{f(X) - u}{1 - y(f(X) - u)}} = 0.
    \]
Note that for $y\geq 0,$ this equation  coincides with the optimality condition for $\lambda^\star$ in the definition of $\Kinf(\up, u, f).$     Since $\up f < u < \ub$, the function $y\mapsto S(\rmi y)=\E\left[\log(1 - y (f(X) - u)) \right]$ is strictly concave in $y$ and, therefore, equation $S'(\rmi y) = 0$ has a unique solution $y=\lambda^\star$. Thus the unique saddle point of $S$ is equal to $z_0 = \rmi\lambda^\star$.
Next, let us change the integration contour to $\gamma^\star = \R + \rmi\lambda^\star$. To prove that this contour is suitable, let us show that the real part of $S$ achieves a minimum at $z_0$ over all $z \in \gamma^\star$
    \[
        \re S(x + \rmi\lambda^\star) = \frac{1}{2} \E\left[ \log\left( (1 - \lambda^\star (f(X) - u))^2 + x^2 (f(X) - u)^2 \right) \right].
    \]
    The minimum of \(\re S(x + \rmi\lambda^\star)\) is achieved for $x=0$, therefore the contour  \(\gamma^\star\) is suitable. Hence, we can apply the Laplace method after a simple change of coordinates
    \[
        I = \int_\R \frac{\exp\left( -\ualpha \,\E\left[\log( 1 - \lambda^\star (f(X) - u) + \rmi s (f(X) - u))\right] \right)}{1-\lambda^\star(f(\ell) - u)+\rmi(f(\ell) - u)s} \, \rmd s
    \]
    Denote 
    \[
    T(s) = \E\left[\log( 1 - \lambda^\star (f(X) - u) + \rmi s (f(X) - u))\right], \quad P(s) = \frac{1}{1-\lambda^\star(f(\ell) - u) + \rmi s (f(\ell) - u)}.
    \]
     Fix  a cut-off parameter $K > 0$  and  define $\kappa_1 = T(-K) - T(0)$, $\kappa_2 = T(K) - T(0)$. 
 Next,  similarly to Chapter~4 (Section~6) by \citet{olver1997asymptotics}, we define the change of  variables $v_1 = T(-s) - T(0), v_2 = T(s) - T(0)$ and the implicit functions $q_1(v_1) = \frac{P(-s)}{T'(-s)}, \, q_2(v_2) = \frac{P(s)}{T'(s)}.$ Using the first order  Taylor expansion, we can write $q_1(v_1) = \frac{P(0)}{\sqrt{2 T''(0) \cdot v_1}} + r_1(v_1), \, q_2(v_2) = \frac{P(0)}{\sqrt{2 T''(0) \cdot v_2}} + r_2(v_2)$. Then we have the following decomposition
    \[
        I = \int_{-\infty}^{\infty} P(s) \exp(-\ualpha\,  T(s))\, \rmd s = \left( P(0) \cdot \sqrt{\frac{2\pi}{\ualpha \, T''(0)}} - R_1(\alpha) + R_2(\alpha) \right) \exp(-\ualpha \, T(0)) + R_3(\alpha),
    \]
    where
    \begin{align*}
        R_1(\alpha) &=  \left(\Gamma\left( \frac{1}{2}, \kappa_1 \, \ualpha \right) + \Gamma\left( \frac{1}{2}, \kappa_2 \, \ualpha \right) \right) \frac{P(0)}{\sqrt{2 T''(0) \, \ualpha}}\CommaBin \\
        R_2(\alpha) &= \int_0^{\kappa_1} e^{-\ualpha v_1} r_1(v_1)\, \rmd v_1 + \int_0^{\kappa_2} e^{-\ualpha v_2} r_2(v_2)\, \rmd v_2, \\
        R_3(\alpha) &= \int_{\R \setminus [-K, K]} P(s) \exp(-\ualpha \, T(s))\, \rmd s,
    \end{align*}
    where $\Gamma(\alpha, x)$ is an upper incomplete gamma function and integration w.r.t. \(v_1,v_2\) is performed over the straight lines connecting the points \(0\) and \(\kappa_1,\kappa_2,\) respectively. Define $\sigma^2 = T''(0)$.
    
    \paragraph*{Term $R_2$}
    We will start from upper bounding on remainder terms in Taylor-like expansions $r_2(v)$
     \begin{align*}
        \vert r_2(v) \vert &= \left\vert \frac{P(s)}{T'(s)} - \frac{P(0)}{\sqrt{2 T''(0) (T(s) - T(0))}}\right\vert \\
        &\leq  P(0) \left\vert \frac{1}{T'(s)} - \frac{1}{\sqrt{2 T''(0) (T(s) - T(0))}}\right\vert  + \frac{\vert P(s) - P(0) \vert }{\vert T'(s) \vert} \\
        &= P(0) \cdot \bar r_2(v) + \tilde r_2(v).
    \end{align*}
    \textit{Upper bound for $\bar r_2(v)$}.

        By Taylor expansion
        \begin{align*}
            T(s) &= T(0) + T'(0) \cdot s + \frac{T''(0)}{2} s^2 + \frac{T'''(\xi_1)}{6} s^3,  &\xi_1 \in (0, s) \\
            T'(s) &= T'(0) + T''(0) s + \frac{T'''(\xi_2)}{2} s^2, &\xi_2 \in (0,s).
        \end{align*}
        Notice that $T'(0) = 0,$  thus
        \begin{align*}
            \vert \bar r_2(v) \vert &= \left\vert \frac{1}{T'(s)} - \frac{1}{\sqrt{2 T''(0) (T(s) - T(0))}}\right\vert \\
            &= \left\vert \frac{\sqrt{T''(0)^2 s^2 +  T''(0) \frac{T'''(\xi_1)}{3} s^3} - T''(0) s - \frac{T'''(\xi_2)}{2} s^2}{[T''(0)s + \frac{T'''(\xi_2)}{2} s^2] \cdot \sqrt{T''(0)^2 s^2 + T''(0) \frac{T'''(\xi_1)}{2} s^3}}\right\vert \\
            &= \left\vert T''(0) \frac{\sqrt{1  + \frac{T'''(\xi_1)}{3T''(0)} s} - 1 - \frac{T'''(\xi_2)}{2T''(0)} s}{s \cdot [T''(0) + \frac{T'''(\xi_2)}{2} s] \cdot \sqrt{T''(0)^2 + T''(0) \frac{T'''(\xi_1)}{3} s}}\right\vert.
        \end{align*}
        Next by applying the Taylor expansion for square root $\sqrt{1+x} -1 = \frac{x}{2} - \frac{x^2}{8(1 + \xi_3)^{3/2}}$ for $ \vert \xi_3 \vert < x,$
        \begin{align*}
            \vert \bar r_2(v) \vert &= \left\vert T''(0) \frac{\frac{T'''(\xi_1)}{6 T''(0)} - \frac{T'''(\xi_2)}{2 T''(0)} - s \cdot \frac{(T'''(\xi_1)/ T''(0))^2}{8 \cdot 9 \cdot (1 + \xi_3)^{3/2}}}{ [T''(0) + \frac{T'''(\xi_2)}{2} s] \cdot \sqrt{T''(0)^2 + T''(0) \frac{T'''(\xi_1)}{3} s}}\right\vert \\
            &= \left\vert \frac{\frac{T'''(\xi_1)}{6} - \frac{T'''(\xi_2)}{2} - s \cdot \frac{(T'''(\xi_1))^2}{8 \cdot 9 \cdot (1 + \xi_3)^{3/2}  \cdot T''(0)} }{ [T''(0) + \frac{T'''(\xi_2)}{2} s] \cdot \sqrt{T''(0)^2 + T''(0) \frac{T'''(\xi_1)}{3} s}}\right\vert \\
            &\leq \frac{ \frac{2}{3} \sup_{u \in (0,s)} \vert T'''(u) \vert + s \cdot \frac{\sup_{u \in (0,s)} \vert T'''(u)\vert^2}{72 \cdot \vert T''(0) \vert } }{ \vert T''(0) + \frac{T'''(\xi_2)}{2} s \vert \cdot \sqrt{\vert T''(0)^2 + T''(0) \frac{T'''(\xi_1)}{3} s \vert}}.
        \end{align*}
        
        Let us analyze the second and third derivatives of $T$
        \begin{align*}
            T''(s) &= \E\left[ \left(\frac{f(X) - u}{1 - \lambda^\star(f(X) - u) + is (f(X) - u)} \right)^2 \right], \\
            T'''(s) &= -2i \E\left[ \left(\frac{f(X) - u}{1 - \lambda^\star(f(X) - u) + is (f(X) - u)} \right)^3 \right].
        \end{align*}
        Define a random variable $Y_s = \frac{f(X) - u}{1 - \lambda^\star(f(X) - u) + is (f(X) - u)},$ then  $T''(s) = \E[Y_s^2]$, $T'''(s) = -2i \E[Y_s^3]$. Let us compute an upper bound on the absolute value of $T'''(s)$
        \[
            \vert T'''(s) \vert \leq 2 \E\left[ \frac{\vert f(X) - u \vert^3}{( (1 - \lambda^\star(f(X) - u))^2 + s^2 (f(X) - u)^2)^{3/2}}  \right] \leq 2 \E[\vert Y_0 \vert^3].
        \]
              By choosing \(1/(2K) = \max\left\{\frac{\ub-u}{1 - \lambda^\star(\ub-u)}, \frac{u}{1+\lambda^\star u} \right\},\) we ensure that
      \( \E[Y_0^2] - s \E[\vert Y_0\vert^3] \geq \frac{1}{2}\E[Y_0^2]\)  for all $0 \leq s < K,$ since 
      \[
            \E[\vert Y_0\vert^3] \leq \max_{j \in \{0,\ldots,m\}} \frac{\vert f(j) - u \vert}{1 - \lambda^\star(f(j) - u)} \E[Y_0^2] \leq \max\left\{\frac{\ub-u}{1 - \lambda^\star(\ub-u)}, \frac{u}{1+\lambda^\star u} \right\} \E[Y_0^2] \leq \frac{1}{2K}\E[Y_0^2].
     \] 
    Hence 
     \begin{eqnarray*}
    \vert \bar r_2(v) \vert &\leq &\frac{\frac{4}{3} \E[\vert Y_0\vert^3] + s \frac{\E[\vert Y_0\vert^3]^2 }{18 \cdot \E[Y_0^2]}}{(\E[Y_0^2] - \E[\vert Y_0\vert^3] s)\cdot \sqrt{\E[Y_0^2]} \cdot \sqrt{\E[Y_0^2] - \E[\vert Y_0 \vert^3] \frac{2s}{3} }}
    \\
    &\leq &\frac{4/3 + 1/36}{ 1/2 \cdot \sqrt{2/3}} \cdot \frac{\E[\vert Y_0\vert^3]}{\E[Y_0^2]^{2}} \leq \frac{49\sqrt{6}}{36 \E[Y_0^2]} \cdot \max\left\{\frac{\ub-u}{1 - \lambda^\star(\ub-u)}, \frac{u}{1+\lambda^\star u} \right\}.
    \end{eqnarray*}
    Next, using the bound
        \[
            \E[Y_0^2] = \sum_{j=0}^m \frac{\alpha_j}{\ualpha}  \left(\frac{f(j)-u}{1 - \lambda^\star(f(j)-u)}\right)^2 \geq \frac{\alphamin}{\ualpha}  \left(\max\left\{\frac{\ub-u}{1 - \lambda^\star(\ub-u)}, \frac{u}{1+\lambda^\star u} \right\}\right)^2,
        \]
       we obtain 
        \[
            \vert \bar r_2(v) \vert \leq \frac{49\sqrt{6}}{36 \sqrt{\sigma^2}} \cdot \sqrt{\frac{\ualpha}{\alphamin}}.
        \]

    \textit{Upper bound for $\tilde r_2(v)$}
    Our next goal is to analyze the second term $\tilde r_2(v)$. We apply Taylor expansions of the form $T'(s) = T''(0) s + T'''(\xi_2) s^2/2$ and $P(s) = P(0) + P'(\eta) s$ to derive
    \begin{align*}
        \tilde r_2(v) &= \left\vert \frac{ P(s) - P(0) }{T'(s)} \right\vert =  \frac{\vert  P'(\eta) \cdot s \vert}{\vert T''(0)s + T'''(\xi_2) s^2/2 \vert} \leq \frac{ \sup_{\eta \in (0,s)} \vert  P'(\eta) \vert}{\vert T''(0) + T'''(\xi_2) s/2 \vert}\cdot
    \end{align*}
    First note that $\vert P'(\eta) \vert$ maximizes at $\eta = 0,$ since
    \[
        \vert P'(\eta) \vert = \frac{\vert f(\ell) - u \vert}{ (1 - \lambda^\star(f(\ell) - u))^2 + \eta^2 (f(\ell) - u)^2}\,.
    \]
    Next by defining a random variable $Y_s = \frac{f(X) - u}{1 - \lambda^\star(f(X) - u) + is (f(X) - u)}$  and due to our choice of $K$ we conclude that 
    \[
        \vert T''(0) + T'''(\xi_2) s/2 \vert \geq \E[Y_0^2] - s \E[\vert Y_0\vert^3] \geq \E[Y_0^2]/2 = \sigma^2/2.
    \]
    It yields
    \[
        \tilde r_2(v) \leq  \frac{2 \vert f(\ell) - u\vert 
        }{(1 - \lambda^\star(f(\ell) - u))^2 \sigma^2} = \frac{2}{(1 - \lambda^\star(f(\ell) - u))\sqrt{\sigma^2}} \sqrt{\frac{\frac{(f(\ell) - u)^2}{(1 - \lambda^\star(f(\ell) - u))^2}}{\E[Y_0^2] }}\, .
    \]
    By a bound
    \[
        \E[Y_0^2] \geq \frac{\alphamin}{\ualpha}  \left(\max\left\{\frac{\ub-u}{1 - \lambda^\star(\ub-u)}, \frac{u}{1+\lambda^\star u} \right\}\right)^2 \geq \frac{\alphamin}{\ualpha} \frac{(f(\ell) - u)^2}{(1 - \lambda^\star(f(\ell) - u))^2}
    \]
    we obtain
    \[
        \tilde r_2(v) \leq \frac{2}{(1 - \lambda^\star(f(\ell) - u))\sqrt{\sigma^2}} \sqrt{\frac{\ualpha}{\alphamin}}.
    \]
    Unifying bounds on $\bar r_2$ and $\tilde r_2$ we obtain
    \[
        \vert r_2(v) \vert \leq \frac{1}{(1 - \lambda^\star(f(\ell) - u))\sqrt{\sigma^2}} \sqrt{\frac{\ualpha}{\alphamin}} \left(2 + \frac{49 \sqrt{6}}{36} \right).
    \]
    A similar bound  also holds for $r_1(v)$ by symmetry.  Set $c_2' =  2 + \frac{49\sqrt{6}}{36},$ then 
       \begin{align*}
            \vert R_2(\alpha) \vert &\leq \frac{c_2'}{\sqrt{(1-\lambda^\star(f(\ell) -u))^2\sigma^2}} \cdot \sqrt{\frac{\ualpha}{\alphamin}}\cdot \left\vert \int_0^{\kappa_2} e^{-\ualpha v} \rmd v + \int_0^{\kappa_1} e^{-\ualpha v} \rmd v \right\vert \\
            &\leq \frac{2c_2'}{\sqrt{\ualpha (1-\lambda^\star(f(\ell)-u))^2 \sigma^2}} \cdot \frac{1 + \exp(-\ualpha  \kappa)}{\sqrt{\alphamin}},
       \end{align*}
     where \(\kappa=\min\{\re \kappa_1,\re \kappa_2\}.\)  
    Using the identity
        \begin{align*}
            \re \kappa_1 = \re \kappa_2 &= \frac{1}{2} \EE{\log\left( \frac{(1 - \lambda^\star (f(X) - u))^2 + K^2 (f(X) - u)^2}{(1 - \lambda^\star (f(X) - u))^2} \right)}= \frac{1}{2} \EE{\log\left( 1 + K^2 Y_0^2 \right)}
        \end{align*}
     and the inequality
     \begin{align*}
            \EE{\log\left( 1 + K^2 Y_0^2 \right)} &= 
             \sum_{j=0}^m \frac{\alpha_j}{\ualpha}\log\left( 1 + K^2 \cdot \left(\frac{f(j)-u}{1 - \lambda^\star(f(j)-u)}\right)^2 \right) \\
             &\geq\frac{\alphamin}{\ualpha} \log\left( 1 + K^2 \cdot \left(\max\left\{\frac{\ub-u}{1 - \lambda^\star(\ub-u)}, \frac{u}{1+\lambda^\star u} \right\}\right)^2 \right) \geq \frac{\alphamin}{\ualpha}  \log(5/4)
        \end{align*}
    we have $\kappa = \re \kappa_2 = \re \kappa_1 \geq c_\kappa \cdot \frac{\alphamin}{\ualpha}$
        with $c_\kappa = 1/2 \cdot \log\left(5/4\right).$ Since $\alphamin > 0$, we also have $\exp(-c_\kappa \alphamin) \leq 1$. 
    Finally setting $c_2 = 4\cdot c_2' = 8 + \frac{49\sqrt{6}}{9},$ we derive the following bound on $R_2$
        \[
            \vert R_2(\alpha) \vert \leq \frac{c_2}{\sqrt{\ualpha (1-\lambda^\star(\ub - u))\sigma^2}} \cdot \frac{1}{\sqrt{\alphamin}}.
        \]

    \paragraph*{Term $R_1$}
    By Theorem 2.4 of \citet{borwein2007uniform} we have the following bound on complex gamma function for any complex $z$ with $\re z > 0$
        \[
            \left\vert \Gamma\left( \frac{1}{2}, z \right) \right\vert \leq \frac{2 e^{-\re z}}{\vert z \vert^{1/2}}. 
        \]
        Therefore,
        \[
            \vert R_1(\alpha) \vert \leq \frac{4 }{\sqrt{2 T''(0) \vert \kappa \vert}} \cdot \frac{\exp(-\ualpha \,  \kappa) }{\ualpha}.
        \]
        Notice that $T''(0) = \sigma^2$, thus
        \[
            \vert R_1(\alpha) \vert \leq \frac{2\sqrt{2}}{\sqrt{\sigma^2 \cdot \ualpha}} \cdot \frac{\exp(-c_\kappa \alphamin) }{(c_\kappa \alphamin)^{1/2}}.
        \]
    \paragraph*{Term $R_3$}
    We start from the bound
    \begin{align*} 
\label{eq: R3-int}
\left\vert \int_{K}^\infty P(s) \exp(-\ualpha T(s))\, \rmd s \right\vert &\leq  \exp(-\ualpha T(0)) \cdot \exp(-\ualpha (T(K) - T(0))) \cdot \sup_{s \in \R}\vert P(s) \vert \\
&\cdot  \int_{K}^\infty \exp(-\ualpha \re [T(s) - T(K)])\, \rmd s.
\end{align*}
Let us start from the analysis of an additional multiplier connected to $P(s)$
\[
    \sup_{s \in R} \vert P(s) \vert = \sup_{s} \sqrt{\frac{1}{(1 - \lambda^\star(f(\ell) - u))^2 + s^2 (f(\ell) - u) }} = \frac{1}{1 - \lambda^\star(f(\ell) - u)}
\]
    Our goal is to bound the last integral. Let us analyze the function under exponent after the change of variables $s \mapsto t+K$
    \begin{align*}
        q(t) &\triangleq \re [T(t+K) - T(K)] = \frac{1}{2} \EE{\log\left( \frac{(1 - \lambda^\star (f(X) - u))^2 + (t+K)^2 (f(X) - u)^2}{(1 - \lambda^\star (f(X) - u))^2 + K^2 (f(X) - u)^2}\right)} \\
        &= \frac{1}{2} \EE{\log\left( 1+  \frac{(t^2 + 2TK) (f(X) - u)^2}{(1 - \lambda^\star (f(X) - u))^2 + K^2 (f(X) - u)^2}\right)}
    \end{align*}
    Define a function $g(j) = \frac{(f(j) - u)^2}{(1 - \lambda^\star (f(j) - u))^2 + K^2 (f(j) - u)} \geq 0$. Then 
 \begin{eqnarray}
 \label{eq:ineq-q}
    \exp(-\ualpha q(t)) = \prod_{j=0}^m \left( \frac{1}{1 + (2tK + t^2) g(j)}\right)^{\alpha_i/2} \leq \prod_{j=0}^m \left( \frac{1}{1 + t^2 g(j)}\right)^{\alpha_i/2}.
\end{eqnarray}
Take a sequence $q_i \in \R_+$ such that $\sum_{i=0}^m q_i^{-1}= 1$. Following ideas of \cite{gotze2019largeball}, apply the generalized Hölder inequlaity
\begin{align*}
    \int_{K}^\infty &\exp(-\ualpha \re [T(s) - T(K)])\, \rmd s = \int_0^\infty \exp(-\ualpha q(t)) \rmd t \\
    &\leq \int_0^\infty  \prod_{j=0}^m \left( \frac{1}{1 + t^2 g(j)}\right)^{\alpha_j/2} \rmd t \leq \prod_{j=0}^m \left( \int_0^\infty \left( \frac{1}{1 + t^2 g(j)}\right)^{\alpha_j q_j/2}\rmd t \right)^{1/q_j}.
\end{align*}
Next we compute integrals exactly and obtain
\[
    \int_{0}^\infty \left(\frac{1}{1+t^2g(j)} \right)^{\alpha_j q_j / 2}\, \rmd t \leq \frac{\sqrt{\pi} \Gamma((\alpha_j q_j - 1)/2)}{2  \Gamma(\alpha_j q_j / 2)} \cdot \frac{1}{\sqrt{g(j)}}
\]
By Theorem 2 of \citet{guo2007necessary} and assumption $\alpha_j q_j > 1$ we have
\[
    \frac{\Gamma((\alpha_j q_j - 1)/2)}{\Gamma(\alpha_j q_j/2)} \leq \frac{\sqrt{2\rme}}{\sqrt{\alpha_j q_j - 1}}.
\]
Finally, we obtain
\begin{align*}
    \int_0^\infty \exp(-\ualpha q(t)) \rmd t \leq \frac{\sqrt{2\rme \pi}}{2} \prod_{j=0}^m \left(\frac{1}{\sqrt{g(j)(\alpha_j q_j - 1)}}\right)^{1/q_j}.
\end{align*}
Now we fix $q_j$ such that $g(j) (\alpha_j q_j - 1) = \tau$, where the constant $\tau$ is determined by $\sum_{j=0}^m q_j^{-1}=1$.  Thus
\[
    \frac{1}{q_j} = \frac{\alpha_j g(j)}{\tau + g(j)}, \quad \sum_{j=0}^m \frac{\alpha_j g(j)}{\tau + g(j)} = 1.
\]
In particular, from the second equality it follows that $\tau \leq \sum_{j=0}^m \alpha_j g(j)$ and $\tau + \max_{k}g(k) \geq \sum_{j=0}^m \alpha_j g(j)$. We notice that since $\alphamin \geq 2$ we have $\tau \geq \frac{1}{2} \sum_{j=0}^m \alpha_j g(j)$.
Next we have to guarantee that $\alpha_j q_j > 1$ but it is trivially true for our choice of $\tau$ since $\alpha_j q_j = (\tau + g(j)) / g(j) > 1$.
Thus we have
\[
    \int_{K}^\infty \exp(-\ualpha \re [T(s) - T(K)])\, \rmd s  \leq \sqrt{ \frac{\rme \pi}{\ualpha \E[g(X)]}}.
\]
Next we want to relate $\E[g(X)]$ and $\sigma^2$. By definition
\[
    \E[g(X)]= 
    \E\left[ \frac{(f(X) - u)^2}{(1-\lambda^\star(f(X)) - u))^2 + K^2 (f(X)-u)^2)} \right] 
\]
By the choice $1/(2K) = \max\{ \frac{\ub - u}{1 - \lambda^\star(\ub - u)}, \frac{u}{1+\lambda^\star u}  \}$ and, as a consequence $K^2 \cdot (f(j) - u)^2 \leq (1-\lambda^\star(f(j) - u))^2 / 4$ for any $j \in \{0,\ldots, m\}$. Thus
\[
    \E[g(X)] \geq \E\left[  \frac{(f(X) - u)^2}{(1-\lambda^\star(f(X)) - u))^2 + (1-\lambda^\star(f(X)) - u))^2 / 4} \right] \geq \frac{4}{5} \cdot \sigma^2.
\]
Finally, we obtain the following bound for the remainder term by symmetry
\[
    \vert R_3(\alpha) \vert \leq \frac{\exp(-\ualpha \Kinf(\up, u, f)) }{1-\lambda^\star(f(\ell) - u)} \cdot \sqrt{\frac{5\rme \pi}{\ualpha \sigma^2}} \exp(-c_\kappa \alphamin).
\]
\end{proof}

\section{Elements of Linear Algebra}\label{app:linear_algebra}

In this section we provide some useful lemmas that simplify computation of Jacobians.
First, recall that each full-rank matrix $A \in \R^{k \times n}$ for $k < n$ has the following decomposition
\begin{equation}\label{eq:matrix_reduced_form}
    A = C \begin{bmatrix}
        I_k & \hat A
    \end{bmatrix} P,
\end{equation}
where $C \in \R^{k \times k}$ is non-degenerate matrix, $\hat A \in \R^{k \times (n-k)}$ is a matrix of coefficients in front of free variables, and $P \in \R^{n \times n}$ is permutation matrix. This decomposition follows directly from the Gauss–Jordan elimination, central matrix in this decomposition corresponds to reduced row echelon form. It is well-known that the central matrix of this decomposition is unique. In this case we may define the matrix $\cL_A \in \R^{n \times (n-k)}$  that maps $\R^{n-k}$ to $\Hset_A = \{x \in \R^n : Ax = 0\}$ as follows
\begin{equation}\label{eq:LA_form}
    \cL_A = P^\top \begin{bmatrix} 
        - \hat A \\
        I_{n-k}
    \end{bmatrix}
\end{equation}
that corresponds to a canonical way to compute basis of $\Hset_A$ by reduced matrix. Indeed, one may check
\[
    A \cL_A = C \begin{bmatrix}
        I_k & \hat A
    \end{bmatrix} P P^\top \begin{bmatrix} 
        - \hat A \\
        I_{n-k}
    \end{bmatrix} = 0
\]
and $\cL_A$ is full-rank, therefore, it is a proper map to $\Hset_A$. Notice that $\cL_A$ depends only on reduced form of matrix $A$ and does not change under any row operation performed on matrix $A$.
Next we show several properties of these matrices.

\begin{lemma}[Jacobian of a linear parametrization]\label{lem:jacob}
    For any non-zero vector $v \in \R^n$ such that $v_0 \not = 0$  and $t \in \R$ a Jacobian of map $\cL_v^t$ is equal to $\norm{v}_2 / \vert v_0 \vert$.
\end{lemma}
\begin{proof}
     Note that the gradient vector does not depend on constant shifts. Thus the gradient matrix is equal to a linear map $\cL_v$. Define a vector $\tilde v = [v_1,\ldots,v_n]$, then the square Jacobian is equal to
     \[
        [\cL_v]^2 = \det(\cL_v^\top \cL_v) = \det\left(\begin{bmatrix}
            \tilde v / v_0 & I_{n-1}
        \end{bmatrix}
        \begin{bmatrix}
            \tilde{v}^\top / v_0 \\
            I_{n-1}
        \end{bmatrix}\right)
        = \det\left(\frac{\tilde v \tilde{v}^\top}{v_0^2} + I_{n-1}\right).
     \]
     This matrix is a rank-one matrix plus identity. Its eigenvalues are equal to $\norm{\tilde v}^2/v_0^2 + 1$ and $n-2$ ones. Thus $[\cL_v]^2 = \norm{\tilde v}^2/v_0^2 + 1 = \norm{v}^2/v_0^2$.
\end{proof}


\begin{lemma}\label{lem:maps_product}
    Let $A \in \R^{k \times n}, c \in \R^{n}$ and $A_c = \begin{bmatrix} A \\ c^\top \end{bmatrix} \in \R^{(k+1)\times n}$ for $k < n$. Then
    \begin{equation}\label{eq:maps_product}
        \cL_{A} \cL_{c^\top \cL_A} = \cL_{A_c}.
    \end{equation}
\end{lemma}
\begin{proof}
    First, we compute $\cL_{A_c}$ in terms of matrices $A$ and vector $c$. Without loss of generality we may assume that $C = I$ and $P = I$. Let us divide vector $c$ on two parts $c^\top = \begin{bmatrix} c_{(1)}^\top & c_{(2)}^\top \end{bmatrix}$, where $c_{(1)} \in \R^k$ and $c_{(2)} \in \R^{n-k}$. Then we can obtain row-elimination procedure of $A_c$ as follows
    \[
         \begin{bmatrix} A \\ c^\top \end{bmatrix} \mapsto \begin{bmatrix}
             I_k & \hat A \\
             c_{(1)}^\top & c_{(2)}^\top 
         \end{bmatrix} \mapsto
         \begin{bmatrix}
             I_k & \hat A \\
             0_k & c_{(2)}^\top - c_{(1)}^\top \hat A
         \end{bmatrix}.
    \]
    Define $\hat c = c_{(2)} - \hat A^\top c_{(1)}$. Without loss of generality we may assume that the first coordinate $\hat c_1 \not = 0$, therefore we may decompose $\hat c = \hat c_1 \begin{bmatrix}
        1 & \tilde c
    \end{bmatrix}$, where $\tilde c$ is a rest of coordinates of of $\hat c$ rescaled by $\hat c_1$. Next, we convert upper-triangular form to reduced as follows
    \[
        \begin{bmatrix}
             I_k & \hat A \\
             0_k & \hat c^\top
         \end{bmatrix} \mapsto
         \begin{bmatrix}
             I_k & \hat A^{(1)} & \hat A^{(2:)} \\
             0_k & 1 & \tilde c^\top
         \end{bmatrix}
          \mapsto
         \begin{bmatrix}
             I_k & 0_{n-k} & \hat A^{(2:)} - \hat A^{(1)} \tilde c^\top \\
             0_k & 1 & \tilde c^\top
         \end{bmatrix},
    \]
    where $\hat A^{(1)}$ is a first column of matrix $\hat A$ and  $\hat A^{(2:)}$ is a rest of the matrix $\hat A$ except the first column. Since we reduce the matrix to reduced form, we obtain
    \[
        \cL_{A_c} = \begin{bmatrix}
           \hat A^{(1)} \tilde c^\top - \hat A^{(2:)} \\
           -\tilde c^\top \\
           I_{n-k-1}
        \end{bmatrix}.
    \]
    Next we are going to show that the left-hand side of the equation \eqref{eq:maps_product} is equal to obtain expression of $\cL_{A_c}$. First, we notice that
    \[
        \cL_A^\top c = \begin{bmatrix}
            - \hat A^\top &
            I_{n-k}
        \end{bmatrix}  \begin{bmatrix}
            c_{(1)} \\
            c_{(2)}
        \end{bmatrix}
         = c_{(2)} - \hat A^\top c_{(1)} = \hat c.
    \]
    Next, we have by explicit row-reduction for a system with one equation
    \[
        \cL_{c^\top \cL_A} = 
        \begin{bmatrix}
            -\tilde c^\top \\
            I_{n-k-1}        
        \end{bmatrix},
    \]
    therefore
    \[
        \cL_A \cL_{c^\top \cL_A} = \begin{bmatrix}
            -\hat A^{(1)} & -\hat A^{(2:)} \\
            1 & 0_{n-k-1} \\
            0_{n-k-1} & I_{n-k-1}
        \end{bmatrix}
        \begin{bmatrix}
            -\tilde c^\top \\
            I_{n-k-1}        
        \end{bmatrix}
        = \begin{bmatrix}
           \hat A^{(1)} \tilde c^\top - \hat A^{(2:)} \\
           -\tilde c^\top \\
           I_{n-k-1}
        \end{bmatrix}.
    \]
\end{proof}

\section{\MTS algorithm}
\label{app:mts_algorithms}

In this section we give a detailed description of the \MTS algorithm in Algorithm~\ref{alg:MTS}. We also depicts in Algorithm~\ref{alg:RMTS} the \ourAlgBounded introduced in Section~\ref{sec:multinomial_ts}.

\begin{algorithm}[h!]
\centering
\caption{\MTS}
\label{alg:MTS}
\begin{algorithmic}[1]
  \STATE {\bfseries Input:} Prior distribution $\rho_a^0$ for all arms $a\in [K]$.
      \FOR{$t \in [T]$}
        \STATE Generate posterior sample $w_a^{t}\sim \rho_a^{t-1}$ for all arms $a\in[K]$.
        \STATE Sample arm $A_t \in \argmax_{a\in[K]} \sum_{i=0}^m w_a^{t} \frac{i}{m}$
        \STATE Get reward $Y_t$ and update posterior distributions.
      \ENDFOR
\end{algorithmic}
\end{algorithm}

\begin{algorithm}[h!]
\centering
\caption{\ourAlgBounded}
\label{alg:RMTS}
\begin{algorithmic}[1]
  \STATE {\bfseries Input:} Grid-size $m$. Prior distribution $\rho_a^0$  over the grid $\{0,1/m,\ldots,1\}$ for all arms $a\in [K]$.
      \FOR{$t \in [T]$}
        \STATE Generate posterior sample $w_a^{t}\sim \rho_a^{t-1}$ for all arms $a\in[K]$.
        \STATE Sample arm $A_t \in \argmax_{a\in[K]} \sum_{i=0}^m w_a^{t} \frac{i}{m}$
        \STATE Get reward $Y_t$ and category $i_t = \lfloor m Y_t\rfloor$.
        \STATE Sample virtual reward $\tY_t=(i_t+b_t)/m$ where $b_t\sim \Ber(m Y_t - i_t)$.
        \STATE Update posterior distributions with $\tY_t$.
      \ENDFOR
\end{algorithmic}
\end{algorithm}

\section{Proof of regret bounds}
\label{app:proof_bound_mts}

In this section we provide proofs of Theorem~\ref{thm:multinomial_ts_regret} and Theorem~\ref{thm:bounded_ts_regret}.

\subsection{Proof of Theorem~\ref{thm:multinomial_ts_regret}}

For $\alpha \in \R^{M+1}_+$ we define $\up(\alpha) \in \simplex_M$ as a probability distribution defined as $\up(\alpha)_i = \alpha_i / (\sum_{j=0}^m \alpha_j)$. Additionally, define two following probability distributions
\[
    \up^-(\alpha) = \up(\alpha_0-1,\alpha_1\ldots,\alpha_{m-1}, \alpha_m), \qquad \up^+(\alpha) = \up(\alpha_0, \alpha_1, \ldots, \alpha_{m-1}, \alpha_m - 1).
\]
Let $\varepsilon > 0$ be an value that will be specified later. Let $a^\star$ be a unique optimal arm.
Define the following events
\begin{align*}
    \cE^{\conc, 1}_\varepsilon(t,a) &= \bigg\{ (N^t_a + \ualpha^0-1)\Kinf(\up^{-}(\alpha^t_a), \mustar - \varepsilon, f) \geq N^t_a \Kinf(p_a, \mustar - \varepsilon, f) - \teps(N^t_a) \bigg\}; \\
    \cE^{\conc, 2}_\varepsilon(t) &= \bigg\{ (N^t_{a^\star} + \ualpha^0-1)\Kinf(\up^{-}(\alpha^t_{a^\star}), 1-\mustar + \varepsilon, 1-f) \\
    &\qquad\qquad\qquad \geq N^t_{a^\star} \Kinf(p_{a^\star}, 1-\mustar+ \varepsilon, 1-f) - \teps(N^t_{a^\star}) \bigg\}; \\
    \cE^{\conc, 3}_\varepsilon(t) &= \bigg\{ (N^t_{a^\star} + \ualpha^0-1)\Kinf(\up^{+}(\alpha^t_{a^\star}), \mustar, f) \leq 3 \log(2N^t_a+1) + (\ualpha^0-1) \log\left( \frac{1}{1-\mustar} \right) \bigg\}; \\
    \cE^{\conc}_\varepsilon(t,a) &= \cE^{\conc, 1}_\varepsilon(t, a) \cap \cE^{\conc, 2}_\varepsilon(t) \cap \cE^{\conc, 3}_\varepsilon(t),
\end{align*}
where $N^t_a$ is a number of pulls of arm $a$ till the moment $t$, $\ualpha^0$ is a weight of a prior distribution, and
\[
    \teps(n) = \sqrt{ \frac{4 \log(n) \gamma}{n}} + \frac{3(\ualpha^0 - 1)}{n}\log\left( \frac{n + \ualpha^0 - 1}{1 - \mu^\star + \varepsilon} \right)
\]
for
\[
    \gamma = \frac{1}{\sqrt{1-\mustar + \varepsilon}} \left( 16\rme^{-2} + \log^2\left( \frac{1}{1-\mustar + \varepsilon} \right) \right).
\]
Let us define the following decomposition for any suboptimal arm $a \not = a^\star$
\begin{align*}
    \E[N^T_a] &= \underbrace{\E\left[\sum_{t=1}^T \ind\{ A^t = a, w^t_{a} f \geq \mustar - \varepsilon, \cE^\conc_\varepsilon(t-1, a) \} \right]}_{\mathrm{PostCV}(\varepsilon)} \\
    &+ \underbrace{\E\left[\sum_{t=1}^T \ind\{ A^t = a, w^t_{a} f \leq \mustar - \varepsilon, \cE^\conc_\varepsilon(t-1, a) \} \right]}_{\mathrm{PreCV}(\varepsilon)} \\
    &+ \underbrace{\E\left[\sum_{t=1}^T \ind\{ A^t = a, \lnot \cE^\conc_\varepsilon(t-1, a) \} \right]}_{\mathrm{Conc}(\varepsilon)}.
\end{align*}
Next we have the following three propositions that described the dependence in $T, \varepsilon$ and $M$.
\begin{proposition}\label{prop:concentration_term}
    For multinomial TS we have for any $\varepsilon > 0$ and any suboptimal arm $a$
    \[
        \mathrm{Conc}(\varepsilon) = \cO(1).
    \]
\end{proposition}

\begin{proposition}\label{prop:post_convergence_term}
    For multinomial TS we have for any $\varepsilon > 0$ and any suboptimal arm $a $
    \[
        \mathrm{PostCV}(\varepsilon) \leq \frac{\log T}{\Kinf(p_a, \mustar, f)} + \cO\left( \varepsilon \cdot \log T + \sqrt{\log(T)} \cdot \log\log(T) \right).
    \]
    In particular, this term does not depend on $M$.
\end{proposition}

\begin{proposition}\label{prop:pre_convergence_term}
    For multinomial TS we have for any $\varepsilon > 0$ and any suboptimal arm $a $
    \[
        \mathrm{PreCV}(\varepsilon) = \cO\left( \varepsilon^{-9} \right).
    \]
\end{proposition}

The proof of these three propositions heavily utilizes Theorem~\ref{thm:bound_dbc} and it is postponed to the end of the section. They implies
\[
    \E[N^T_a] \leq \frac{\log T}{\Kinf(p_a, \mustar, f)} + \cO\left( \varepsilon \cdot \log T + \sqrt{\log(T)} \cdot \log\log(T) + \varepsilon^{-9} \right).
\]
Taking $\varepsilon = \cO(\log^{-1/10}(T))$ we conclude the statement.

\subsection{Proof of Theorem~\ref{thm:bounded_ts_regret}}

Lemma 6 by \cite{riou20a} implies for any $m > 1/(1-\mustar)$ the following holds
\[
    \Kinf(\nu_a, \mustar) - \frac{1}{m(1-\mustar) - 1} \leq \Kinf^{(m)}(\nu_a^{(m)}, \mustar) \leq \Kinf(\nu_a, \mustar).
\]
Additionally, Theorem~\ref{thm:multinomial_ts_regret} implies that for \ourAlgBounded with any $m > \max_{a \not = a^\star} \frac{1 + \Kinf(\nu_a, \mustar)}{ (1-\mustar) \cdot \Kinf(\nu_a, \mustar)}$
\begin{align*}
    \E[N^T_a] &\leq \frac{\log T}{\Kinf^{(m)}(\nu^{(m)}_a, \mustar)} + \cO\left( \log^{9/10}(T) \right) \\
    &\leq \frac{\log T}{\Kinf(\nu_a, \mustar)} \cdot \frac{1}{1 - [\Kinf(\nu_a, \mustar) \cdot( m(1-\mustar) - 1) ]^{-1})}  + \cO\left( \log^{9/10}(T) \right).
\end{align*}
Finally, for $x \in (0, 1/2)$ we have $1/(1-x) \leq 1 + 2x$, thus for $m > \max_{a \not = a^\star} \frac{2 + \Kinf(\nu_a, \mustar)}{ (1-\mustar) \cdot \Kinf(\nu_a, \mustar)} $
\[
    \E[N^T_a] \leq  \frac{\log T}{\Kinf(\nu_a, \mustar)} + \frac{2 \log T}{ (\Kinf(\nu_a, \mustar))^2} \cdot\ \frac{1}{m \cdot (1-\mustar) - 1}  + \cO\left( \log^{9/10}(T) \right).
\]
Taking $m \geq \log(T)$ for $T$ large enough we conclude the statement. Notably, for $T$ large enough all conditions on $m$ are satisfied.

\subsection{Proof of Proposition~\ref{prop:concentration_term}-\ref{prop:pre_convergence_term}}
In this section we gather the
proofs of Proposition~\ref{prop:concentration_term}, Proposition~\ref{prop:post_convergence_term} and Proposition~\ref{prop:pre_convergence_term}.

\subsubsection{Proof of Propositions~\ref{prop:concentration_term}}
    By union bound we have
    \[
        \mathrm{Conc}(\varepsilon) \leq \sum_{t=1}^T \P[\lnot \cE^{\conc,1}_\varepsilon(t-1,a)] + \sum_{t=1}^T \P[\lnot \cE^{\conc,2}_\varepsilon(t-1)] + \sum_{t=1}^T \P[\lnot \cE^{\conc,3}_\varepsilon(t-1)].
    \]
    Let us start from the bound on the first term. First, we notice that the distribution of $\alpha^t_a$ depend only on $N^t(a)$ and the distribution $p_a$. Thus, if we fix $N^t(a) = n$ for an arm $a$, we fix the distribution of $\alpha^t_a$ for any arm under this condition. We will call the corresponding random variable as $\alpha^{[n]}_a$.
    By a union bound we have
    \[
        \sum_{t=1}^T \P[\lnot \cE^{\conc,1}_\varepsilon(t-1,a)] \leq  1 + \sum_{n=1}^T \E\left[ \sum_{t=0}^{T-1} \ind\{ \lnot \widetilde{\cE^{\conc}_\varepsilon}(N^t_a, a), N^t_a=n \} \right],
    \]
    where 
    \[
        \widetilde{\cE^{\conc,1}_\varepsilon}(n, a) = \left\{ (n+\ualpha^0-1) \Kinf(\up^{-}(\alpha^{[n]}_a), \mustar - \varepsilon, f) \geq n (\Kinf(p_a, \mustar - \varepsilon, f) - \teps(n))  \right\}.
    \]
    By Lemma~\ref{lem:kinf_concentration} for our choice of $\teps(n)$ we have
    \[
        \sum_{t=1}^T \P[\lnot \cE^{\conc,1}_\varepsilon(t-1,a)] \leq 1 + \sum_{n=1}^T \P[\lnot \widetilde{\cE^{\conc,1}_\varepsilon}(n, a) ] \leq 1 + \sum_{n=1}^T \frac{1}{n^2} \leq 1 + \frac{ \pi^2}{6}.
    \]
    For the second term we have exactly the same argument and we omit it. For the third term we use exactly the same trick for changing summation over $t$ to summation over $n$
    \[
        \sum_{t=1}^T \P[\lnot \cE^{\conc,3}_\varepsilon(t-1)] \leq 1 + \sum_{n=1}^T \P[\lnot \widetilde{\cE^{\conc,3}_\varepsilon}(n)],
    \]
    where 
    \[
        \widetilde{\cE^{\conc,3}_\varepsilon}(n) = \left\{ (n+\ualpha^0-1) \Kinf(\up^{+}(\alpha^{[n]}_{a^\star}), \mustar, f) \leq 3\log(2n+1) + (\ualpha^0-1) \log\left( \frac{1}{1-\mustar} \right)  \right\}.
    \]
    Applying Lemma~\ref{lem:kinf_deviation} we obtain
    \[
        \sum_{t=1}^T \P[\lnot \cE^{\conc,3}_\varepsilon(t-1)] \leq 1 + \rme \sum_{n=1}^T \frac{1}{(2n+1)^2} \leq 1 + \frac{\rme \pi^2}{8}.
    \]

\subsubsection{Proof of Proposition~\ref{prop:post_convergence_term}}

\begin{proof}
    Define a constant $n_0$ that will be specified later. Then we have the following
    \begin{align*}
        \mathrm{PostCV}(\varepsilon) &\leq n_0  + \E\left[\sum_{t=1}^T \ind\{ A^t = a, w^t_{a} f \geq \mustar - \varepsilon, \cE^{\conc,1}_\varepsilon(t-1, a), N^{t-1}_a \geq n_0 \} \right] \\
        &\leq n_0 + \sum_{t=0}^{T-1} \sum_{n=n_0}^T \P\left[ w^{t+1}_a f \geq \mustar - \varepsilon , N^{t}_a = n,  \cE^{\conc,1}_\varepsilon(t,a)  \right],
    \end{align*}
    Next we notice that $N^{t}_a$ and $\ind\{\cE^{\conc,1}_\varepsilon(t,a)\}$ are functions of $\alpha^{t}_a$,    thus, by the tower property of conditional expectation
    \[
        \P\left[ w^{t+1}_a f \geq \mustar - \varepsilon , N^{t}_a = n,  \cE^{\conc,1}_\varepsilon(t,a)  \right] = \E\left[ \P[w^{t+1}_a f \geq \mustar - \varepsilon | \alpha^{t}_a] \ind\{ N^{t}_a = n,  \cE^{\conc,1}_\varepsilon(t,a)\}  \right],
    \]
    where we can assume that $\alpha^{t}_a$ satisfied the definition of $\cE^{\conc,1}_\varepsilon(t,a)$ and $N^{t}_a = n$. 

    Next we notice that
    \[
        \P[w^{t+1}_a f \geq \mustar - \varepsilon | \alpha^t_a] = \P_{w \sim \Dir(\alpha^t_a)}[ w f \geq \mustar - \varepsilon].
    \]
    and we automatically have $\ualpha = n + \ualpha^0$. Next, we notice that if $\up^{(-)}(\alpha^t_a) f > \mustar - \varepsilon$, then the required probability is bounded by $1$. Otherwise we can apply Theorem~\ref{thm:bound_dbc}
    \[
        \P_{w \sim \Dir(\alpha^t_a)}[ wf \geq \mustar - \varepsilon] \leq 2 \P_{g \sim \cN(0,1)}\left[ g \geq \sqrt{2 (n + \ualpha^0-1) \Kinf(\up^{-}(\alpha^t_a), \mustar - \varepsilon, f)} \right].
    \]
    Since $\Kinf$ is non-negative, this upper bound also holds if $\up^{(-)}(\alpha^t_a) f > \mustar - \varepsilon$. Thus, on event $\cE^{\conc,1}_{\varepsilon}(t,a)$ for $N^t_a = n$ we have
    \begin{align*}
        &\P_{g \sim \cN(0,1)}\left[ g \geq \sqrt{2 (n + \ualpha^0-1) \Kinf(\up^{-}(\alpha^t_a), \mustar - \varepsilon, f)}  \right] \\
        &\qquad\qquad\leq \P_{g \sim \cN(0,1)}\left[ g \geq \sqrt{2n (\Kinf(p_a, \mustar - \varepsilon, f) - \teps(n))  }  \right].
    \end{align*}
    Overall, we obtain for any $t = 0,\ldots,T-1$
    \begin{align*}
        \P\left[ w^{t+1}_a u \geq \mustar - \varepsilon |  \alpha^t_a  \right] & \ind\{N^t_a=n, \cE^{\conc,1}_\varepsilon(t,a) \} \\ 
        &\leq 2\P_{g \sim \cN(0,1)}\left[ g \geq \sqrt{2n (\Kinf(p(a), \mustar - \varepsilon, f) - \teps(n))  }  \right] \ind\{N^t_a=n\}.
    \end{align*}
    The right-hand side of the probability strictly increasing in $n$ since $n = \Omega(1)$ we can guarantee $\Kinf(p(a), \mustar - \varepsilon, f) - \teps(n) \geq 1/2 \Kinf(p(a), \mustar - \varepsilon, f)$.
    Therefore we can obtain a uniform upper bound for all $n \geq n_0$ that implies
    \[
        \mathrm{PostCV}(\varepsilon) \leq n_0 + 2T \cdot \P_{g \sim \cN(0,1)}\left[ g \geq \sqrt{2 n_0 (\Kinf(p, \mustar - \varepsilon, f) - \teps(n_0))  } \right].
    \]
    Next, we can use the following upper bound on tails of a normal law, for any $x > 0$ it holds (see e.g. \cite{vershynin2018high})
    \[
         \P_{g \sim \cN(0,1)}\left[ g \geq \sqrt{2x} \right] \leq \frac{\exp(-x)}{2 \sqrt{\pi x}}.
    \]
    Thus, taking $n_0 = (1 + \varepsilon') \log T / \Kinf(p, \mustar - \varepsilon, u)$ for any $\varepsilon' > 0$ we have
    \[
        \mathrm{PostCV}(\varepsilon) \leq \frac{(1 + \varepsilon')\log T}{\Kinf(p_a, \mustar - \varepsilon, f)}  + 2T \cdot \frac{\exp\left( -(1+ \varepsilon')\log T + \cO\left( \sqrt{\log\log(T) \cdot \log(T)} \right) \right)}{\sqrt{2 \pi (1 + \varepsilon') \log(T) }}
    \]
    In particular, we can take $\varepsilon' = \log\log(T) / \sqrt{\log(T)}$ to make the last term vanishing for $T$ large enough. Finally, by using Lemma~11 by \cite{garivier2018kl} we have
    \[
        \mathrm{PostCV}(\varepsilon) \leq \frac{\log T}{\Kinf(p_a, \mustar, f)} + \cO\left( \varepsilon \cdot \log T + \sqrt{\log(T)} \cdot \log\log(T) \right).
    \]
\end{proof}

\subsubsection{Proof of Proposition~\ref{prop:pre_convergence_term}}

\begin{proof}
We denote by $\cF^{t}$ the information available at the end of round $t$. First notice, introducing the sub-optimal arm with the largest index $\tA^t = \max_{a'\neq\astar} w_{a'}^t f$, it holds by independence of the posterior samples
\begin{align*}
    \P[A^t = \astar, w_a^t f \leq \mustar-\varepsilon |\cF^{t-1}] &\geq \P[\tA^t = a, w_a^t f \leq \mustar-\varepsilon, w_{\astar}^t f > \mustar-\varepsilon  |\cF^{t-1}]\\
    &= \big(1-\P[w_{\astar}^t f > \mustar-\varepsilon|\cF^{t-1}]\big) \P[\tA^t = a, w_a^t f \leq \mustar-\varepsilon |\cF^{t-1}]
\end{align*}
and 
\begin{align*}
    \P[A^t = a, w_a^t f \leq \mustar-\varepsilon |\cF^{t-1}] &\leq \P[\tA^t = a, w_a^t f \leq \mustar-\varepsilon, w_{\astar}^t f \leq \mustar-\varepsilon  |\cF^{t-1}]\\
    &= \P[\tA^t = a, w_{a}^t f \leq \mustar-\varepsilon|\cF_{t-1}]\P[w_{\astar}^t f \leq  \mustar-\varepsilon|\cF^{t-1}]\,.
\end{align*}
From these two inequalities we deduce that 
\begin{align*}
    \P[A^t = a, w_a^t f \leq \mustar-\varepsilon |\cF^{t-1}] \leq \frac{\P[w_{\astar}^t f \leq  \mustar-\varepsilon|\cF^{t-1}]}{1-\P[w_{\astar}^t f \leq  \mustar-\varepsilon|\cF_{t-1}]}  \P[A^t = \astar, w_a^t f \leq \mustar-\varepsilon |\cF^{t-1}]\,.
\end{align*}
By conditioning on $\cF^{t-1}$ and using the previous inequality, then switching from global time to counts we obtain 
\begin{align*}
\mathrm{PreCV}(\varepsilon) &= \E\left[\sum_{t=1}^T \ind\{ A^t = a, w^t_{a} f \leq \mustar - \varepsilon, \cE^\conc_\varepsilon(t-1, a) \} \right]\\
&\leq \E\left[ \sum_{t=1}^T
\frac{\P[w_{\astar}^t f \leq  \mustar-\varepsilon|\cF^{t-1}]}{1-\P[w_{\astar}^t f \leq  \mustar-\varepsilon|\cF_{t-1}]}  \ind\{A^t = \astar, \cE^{\conc,2}_\varepsilon(t-1) \cap \cE^{\conc,3}_\varepsilon(t-1)\}
\right]\\
&\leq 
\E\left[ \sum_{n=0}^{T-1}
 \frac{\P_{w \sim \Dir(\alpha^{(n)}_{\astar})}\left[ w f \leq \mustar - \varepsilon\right]}{1 - \P_{w \sim \Dir(\alpha^{(n)}_{\astar})}\left[ w f \leq \mustar - \varepsilon\right]} \ind\{\widetilde{\cE^{\conc,2}_\varepsilon}(n) \cap \widetilde{\cE^{\conc,3}_\varepsilon}(n)\}
\right]
\end{align*}
where $\alpha^{(n)}_{\astar}$ is the parameter of the the posterior of the optimal arm after $n$ observations and the events $\widetilde{\cE^{\conc,2}_\varepsilon}(n)$, $\widetilde{\cE^{\conc,3}_\varepsilon}(n)$ defined in the proof of Proposition~\ref{prop:concentration_term} are such that 
  \[
        \{ \cE^{\conc,2}_\varepsilon(\tau_n) \cap \cE^{\conc,3}_\varepsilon(\tau_n) \} = \{\widetilde{\cE^{\conc,2}_\varepsilon}(n) \cap \widetilde{\cE^{\conc,3}_\varepsilon}(n)\},
    \]
with $\tau_n$ the first instant such that $N_{\astar}^{\tau_n} = n$.

    To work with this expectation properly, we separate it into two cases
    \[
        \termA = \sum_{n=0}^T \E\left[  \frac{\P_{w \sim \Dir(\alpha^{(n)}_{\astar})}\left[ w f \leq \mustar - \varepsilon \right]}{\P_{w \sim \Dir(\alpha^{(n)}_{\astar})}\left[ w f > \mustar - \varepsilon \right]} \cdot \ind\{\widetilde{\cE^{\conc,3}_\varepsilon}(n), \,\up(\alpha^{(n)}_{\astar}) \cdot f \leq \mustar - \varepsilon \} \right],
    \]
    and
    \[
        \termB = \sum_{n=0}^T \E\left[  \frac{\P_{w \sim \Dir(\alpha^{(n)}_{\astar})}\left[ w f \leq \mustar - \varepsilon \right]}{\P_{w \sim \Dir(\alpha^{(n)}_{\astar})}\left[ w f > \mustar - \varepsilon \right]} \cdot \ind\{\widetilde{\cE^{\conc,2}_\varepsilon}(n),\,\up(\alpha^{(n)}_{\astar}) \cdot f > \mustar - \varepsilon \} \right].
    \]
    Our goal is to bound these two terms separately.
    \paragraph*{Term $\termA$} For this term we have
    \[
        \P_{w \sim \Dir(\alpha^{(n)}_{\astar})}\left[ w f \leq \mustar - \varepsilon \right] \leq 1
    \]
    and by Theorem~\ref{thm:bound_dbc},
    \begin{align*}
        \P_{w \sim \Dir(\alpha^{(n)}_{\astar})}\left[ w f > \mustar - \varepsilon\} \right] &\geq \frac{1}{2} \P_{g \sim \cN(0,1)}\left[ g \geq \sqrt{2(n+\ualpha^0-1) \Kinf(\up^{+}(\alpha^{(n)}_{\astar}), \mustar - \varepsilon, f) }\right] \\
        &\geq \frac{1}{2} \P_{g \sim \cN(0,1)}\left[ g \geq \sqrt{2(n+\ualpha^0-1) \Kinf(\up^{+}(\alpha^{(n)}_{\astar}), \mustar, f) }\right].
    \end{align*}
    By the definition of event $\widetilde{\cE^{\conc,3}_\varepsilon}(n)$ we have 
    \begin{align*}
        (n+\ualpha^0-1)\Kinf(\up^{+}(\alpha^{(n)}_{\astar}), \mustar, f) &\leq 3\log(2n+1) + (\ualpha^0-1) \log\left(\frac{1}{1-\mustar}\right).
    \end{align*}
    Notice that for $n\geq 1$ we have $3\log(2n+1) \geq 1$, thus the lower bound on Gaussian tails implies
    \[
        \P_{w \sim \Dir(\alpha^{(n)}_{\astar})}\left[ w f > \mustar - \varepsilon\} \right]  \geq \frac{(2n+1)^{-3} \cdot (1-\mustar)^{\ualpha^0-1}}{4\sqrt{2} \sqrt{ 3\log(2n+1)+  (\ualpha^0-1) \log\left(\frac{1}{1-\mustar}\right)} }.
    \]
    Therefore, we obtain the following bound
    \begin{align*}
        \termA &\leq \sum_{n=0}^T \cO(n^3 \log^{1/2}(n)) \E\left[ \ind\{\up(\alpha_{\astar}^{(n)}) \cdot f \leq \mustar - \varepsilon \} \right]  \\
        &\leq \sum_{n=0}^T \cO(n^3 \log^{1/2}(n)) \P_{\alpha \sim \mathrm{Mult}(n, p_{a^\star})}\left[ \up(\alpha + \alpha^0) \cdot f \leq \mustar - \varepsilon \right].
    \end{align*}
    To upper bound other probability we use Hoeffding inequality
    \begin{align*}
        \P_{\alpha \sim \mathrm{Mult}(n, p_{a^\star})}&\left[ \up(\alpha + \alpha^0) \cdot u \leq \mustar - \varepsilon  \right] \\
        &\leq \P_{\alpha \sim \mathrm{Mult}(n, p_{a^\star})}\left[ \vert \langle \alpha,u \rangle - n \mustar \vert \geq n \varepsilon - \ualpha^0(\mustar - \langle \alpha^0, u\rangle - \varepsilon )   \right] \\
        &\leq \exp\left(-2 n \left(\varepsilon - \frac{\ualpha^0}{n}(\mustar - \langle \alpha^0, u\rangle - \varepsilon )\right)^2 \right).
    \end{align*}
    Define a truncation point $n_0' = 4 \ualpha^0 \mustar / \varepsilon$. For $n \geq n'_0$ we have
    \[
        \P_{\alpha \sim \mathrm{Mult}(n, p_{a^\star})}\left[ \up(\alpha + \alpha^0) \cdot u \leq \mustar - \varepsilon  \right]\leq \exp\left(\frac{-n \varepsilon^2}{2}\right).
    \]
    And for $n \leq n'_0$ we have a trivial upper bound by $1$. Therefore we obtain
    \[
        \termA \leq \cO\left( \frac{\log^{1/2}(1/\varepsilon)}{\varepsilon^4} \right) + \cO\left(  \sum_{n=n_0'}^T n^3 \log^{1/2}(n) \rme^{-n \varepsilon^2/2} \right).
    \]
    Next, we want to choose next truncation $n_0''$ such that
    \[
        (n_0'')^{3.5} \leq \exp(n_0'' \varepsilon^2 / 4) \iff  14 \log(n_0'') \leq n_0'' \cdot \varepsilon^2.
    \]
    By inequality $\log(x) \leq \beta \cdot x^{1/\beta} $ for any $x>0$ and $\beta > 0$ we can obtain the value of $n_0''$ that indeed satisfy the inequality above $n_0'' =  (14\beta)^{\beta / (\beta - 1)} \cdot \varepsilon^{-2\beta/(\beta - 1)}$ for any constant $\beta > 0$. Thus, we have
    \[
        \termA  \leq \cO\left( \frac{\log^{1/2}(1/\varepsilon)}{\varepsilon^4} \right) + \cO\left( \varepsilon^{-4\beta / (\beta - 1)} \log^{1/2}(\varepsilon) \right) + \cO\left(  \sum_{n=n_0''}^T \rme^{-n \varepsilon^2/4} \right).
    \]
    Taking $\beta = 17/9$ we can see that the second term is dominating in $\varepsilon$ and, moreover, using an inequality $\log^{1/2}(1/\varepsilon)\leq 2\varepsilon^{1/2}$ we have
    \[
        \termA = \cO\left( \varepsilon^{-9} \right).
    \]

    \paragraph*{Term $\termB$}
    For analysis of this term we again start from the probabilities under expectation. First, we have by Theorem~\ref{thm:bound_dbc},
    \[
         \P_{w \sim \Dir(\alpha_{\astar}^{(n)})}\left[ w f > \mustar - \varepsilon \right] \geq  \P_{w \sim \Dir(\alpha_{\astar}^{(n)}))}\left[ w f > \up^+(\alpha^\tau) \cdot f \right] \geq \frac{1}{2}
    \]
    and additionally we have still by Theorem~\ref{thm:bound_dbc}
    \begin{align*}
        \P_{w \sim \Dir(\alpha_{\astar}^{(n)})}\left[ w f \leq \mustar - \varepsilon \right] &= \P_{w \sim \Dir(\alpha_{\astar}^{(n)})}\left[ w (1-f) \geq 1- \mustar + \varepsilon \right] \\
        &\leq \frac{3}{2} \P_{g \sim \cN(0,1)} \left[g \geq \sqrt{2(n+\ualpha^0-1) \Kinf(\up^-(\alpha_{\astar}^{(n)})), 1-\mustar + \varepsilon, 1-f)  } \right].
    \end{align*}
    By definition of event $\widetilde{\cE^{\conc,2}_\varepsilon}(n)$ and Lemma~11 by \cite{garivier2018kl} we have
    \[
        \P_{w \sim \Dir(\alpha_{\astar}^{(n)})}\left[ w f \leq \mustar - \varepsilon \right] \leq \frac{3}{2} \P_{g \sim \cN(0,1)}\left[ g \geq \sqrt{2n (2\varepsilon^2 - \teps(n)) } \right].
    \]
    Define truncation point $\tilde n_0$ such that $\varepsilon^2 \geq \teps(\tilde n_0)$ that is automatically satisfied by choosing $n_0 = \cO(\varepsilon^{-5})$. Thus by an upper bound for tails of normal distribution we obtain
    \[
        \termB \leq \cO(\varepsilon^{-5}) + \sum_{n=\tilde{n}_0}^T \frac{\exp(-n \varepsilon^2)}{\sqrt{n \varepsilon^2}}  = \cO\left( \varepsilon^{-5} \right).
    \]
    Overall, we obtain
    \[
        \mathrm{PreCV}(\varepsilon) = \cO\left( \varepsilon^{-9} \right).
    \]
    
\end{proof}

\section{Technical Lemmas}


First we start from proving an auxiliary result similar to Lemma D.14 of \cite{tiapkin2022dirichlet}
\begin{lemma}\label{lem:ub_lambda_star}
       Let $f \colon \{0,\ldots,m\} \to [0, \ub]$ be a function such that $f(0) = 0$ and $f(m) = b$. Let $(\alpha_0, \ldots, \alpha_m) \in \R^{m+1}_+$ be a vector such that $\alpha_m > 0$ and define $\ualpha = \sum_{i=0}^m \alpha_i$. Define a distribution $\up \in \simplex_m$ such that $\up(i) = \alpha_i / \ualpha$ and consider a real number $\mu \in (\up f, \ub)$  Then
       \[
            1 - \lambda^\star(\up, \mu, f) \cdot (\ub - \mu) \geq \up(m),
       \]
       where $ \lambda^\star(\up, \mu, f) $ is defined as an optimal dual variables to the variational formula
       \[
         \lambda^\star(p, \mu, f) = \argmax_{\lambda \in [0, 1/(\ub - \mu)]} \E_{X \sim p} \left[ \log\left(1 + \lambda \cdot (f(X) - \mu) \right) \right].
       \]
\end{lemma}
\begin{proof}
    Under the conditions $\up f < \mu < \ub,$ and $\up(m) > 0$ the value $\lambda^\star = \lambda^\star(\up, \mu, f)$ satisfies the following equation
    \begin{equation}\label{eq:expectation|lem:ub_lambda_star}
        \E\left[ \frac{f(X) - \mu}{1 - \lambda^\star \cdot (f(X) - \mu)} \right] = \up(m) \frac{\ub - \mu}{1 - \lambda^\star \cdot (\ub - \mu)} + \sum_{j=0}^{m-1} \up(j) \frac{f(j) - \mu}{1 - \lambda^\star \cdot (f(j) - \mu)} = 0.
    \end{equation}
    Define a distribution $\tp$ with $\tp(i) = \up(i) / (1 - \up(m))$ for $i \in \{0,\ldots,m-1\}$ and $\tp(m) = 0$. Then the expectation in \eqref{eq:expectation|lem:ub_lambda_star} can be written as
    \[
        \up(m) \frac{\ub - \mu}{1 - \lambda^\star \cdot (\ub - \mu)} + (1 - \up(m)) \E_{X \sim \tp}\left[\frac{f(X) - \mu}{1 - \lambda^\star \cdot (f(X) - \mu)}\right] = 0.
    \]
    Define a function $w(x,u) = \frac{x-u}{1 - \lambda^\star \cdot (x-u)}$, which is convex  in $x.$ By the Jensen inequality, 
    \[
        \E_{X \sim \tp}\left[\frac{f(X) - \mu}{1 - \lambda^\star \cdot (f(X) - \mu)}\right] \geq \frac{\tp f - \mu}{1 - \lambda^\star \cdot (\tp f - \mu)}.
    \]
 Hence
    \begin{align*}
        \up(m) \frac{\ub - \mu}{1 - \lambda^\star \cdot (\ub - \mu)} &\leq - (1 - \up(m)) \frac{\tp f - \mu}{1 - \lambda^\star \cdot (\tp f - \mu) } =  (1 - \up(m))\frac{\mu - \tp f}{1 + \lambda^\star \cdot (\mu - \tp f) }.
    \end{align*}
    Notice that $\mu \geq \up_n f \geq \tp_n f$, therefore we can rearrange terms as follows
    \begin{align*}
        \frac{1}{\up(m)} \left( \frac{1}{\ub-\mu} - \lambda^\star \right) &\geq \frac{1}{1 - \up(m)} \left( \frac{1}{\mu - \tp f} + \lambda^\star \right) \\
        &= \frac{1}{1 - \up(m)}\left( \frac{1}{\mu - \tp f} + \frac{1}{\ub-\mu}\right) - \frac{1}{1 - \up(m)} \left(\frac{1}{\ub-\mu} - \lambda^\star \right).
    \end{align*}
    As a result,
    \[
        \left( \frac{1}{\ub-\mu} - \lambda^\star \right) \geq \up(m) \cdot \frac{\ub - \hp f}{(\ub-\mu)(\mu - \hp f)} \geq \up(m)\frac{1}{\ub-\mu}.
    \]
\end{proof}

Next we analyze how $\Kinf$ changes with  small changes in  the first argument (distribution). For any non-negative vector $s \in \R^{m+1}_+$, we define $\up_s $ as $\up_{s}(i) = s_i / (\sum_{i=0}^m s_i)$. 

\begin{lemma}\label{lem:distribution_shift}
    Let $f \colon \{0,\ldots,m\} \to [0, \ub]$ be a function such that $f(0) = 0$ and $f(m) = b$. Let $\alpha= (\alpha_0, \ldots, \alpha_m) \in \R^{m+1}_+$ and $\beta = (\beta_0, \ldots, \beta_m) \in \R^{m+1}_+$ be non-negative vectors such that $\beta_m + \alpha_m \geq 1.$ Set $\ualpha = \sum_{i=0}^m \alpha_i$ and $\ubeta = \sum_{i=0}^m \beta_i,$ 
    then the following bounds hold for any $\mu \in (0, b),$
    \[
        \Kinf(\up_\alpha, \mu, f) - \frac{3 \ubeta}{\ualpha} \log\left( \frac{\ub  (\ualpha + \ubeta)}{\ub - \mu} \right) \leq \Kinf(\up_{\alpha + \beta}, \mu, f)\leq \frac{\ualpha \cdot  \Kinf(\up_\alpha, \mu, f) + \ubeta \cdot \log\left( \frac{\ub}{\ub - \mu} \right)}{\ualpha + \ubeta}.
    \]
\end{lemma}
\begin{proof}
    We start from the lower bound.  Let $q^\star$ be a probability measure that attains in the definition of $\Kinf(\up_{\alpha + \beta}, \mu, f),$ that is,
    \[
        q^\star = \argmin_{q \in \simplex_{m}}\{ \KL(\up_{\alpha + \beta}, q)  : qf \geq \mu \}.
    \]
    Then we have
    \begin{align*}
        \Kinf(\up_{\alpha + \beta}, \mu, f) = \KL(\up_{\alpha + \beta}, q^\star) &= \underbrace{\sum_{i: \up_{\alpha + \beta}(i) > 0} (\up_{\alpha + \beta}(i) - \up_\alpha(i)) \log\left( \frac{\up_{\alpha + \beta}(i)}{q^\star(i)} \right)}_{\termA} \\
        &+ \underbrace{\sum_{i: \up_{\alpha + \beta}(i) > 0} \up_\alpha(i)  \log\left( \frac{\up_{\alpha + \beta}(i)}{q^\star(i)} \right)}_{\termB}.
    \end{align*}
    \paragraph*{Term $\termA$}
    Let an analyze the first sum. Since $\up_{\alpha + \beta}(m+1) > 0$, there is a formula for the optimal $q^\star$ that follows from the dual representation of $\Kinf$
    \[
        q^\star(i) = \frac{\up_{\alpha + \beta}(i)}{1 -\lambda^\star \cdot (f(i) - \mu)},
    \]
    where $\lambda^\star \in [0, 1/(\ub - \mu))$ is a solution to the variational problem for $\Kinf(\up_{\alpha + \beta}, \mu, f)$. Thus, we have
    \[
        \sum_{i: \up_{\alpha + \beta}(i) > 0} (\up_{\alpha + \beta}(i) - \up_\alpha(i)) \log\left( \frac{\up_{\alpha + \beta}(i)}{q^\star(i)} \right) \geq - \norm{\up_{\alpha + \beta} - \up_\alpha}_1 \cdot \norm{\log\left(1 - \lambda^\star \cdot (f - \mu) \right) }_\infty.
    \]
    For the first multiplier we have the following bound
    \[
         \norm{\up_{\alpha + \beta} - \up_{\alpha}}_1 = \left\Vert \frac{\ualpha}{\ualpha + \ubeta}\up_\alpha + \frac{1}{\ualpha + \ubeta} \beta - \up_{\alpha} \right\Vert_1 = \frac{1}{\ualpha + \ubeta} \left\Vert \beta - \ubeta \cdot\up_{\alpha} \right\Vert_1 \leq \frac{2 \ubeta }{ \ualpha + \ubeta}.
    \]
   To analyze the second multiplier, we notice that the maximum could be attained only for $i=0$ or $i=m$ since logarithm is a monotone function and $\lambda^\star$ is fixed. Thus
    \[
        \max_{i\in\{0,\ldots,m\}} \left\vert \log\left(1 - \lambda^\star \cdot (f(i) - \mu) \right) \right\vert = \max\left\{ \log(1 + \lambda^\star \mu), \log\left(\frac{1}{1 - \lambda^\star(\ub - \mu)}\right) \right\}.
    \]
By applying Lemma~\ref{lem:ub_lambda_star}, we get
    \[
        \max\left\{ \log(1 + \lambda^\star \mu), \log\left(\frac{1}{1 - \lambda^\star(\ub - \mu)}\right) \right\} \leq \max\left\{ \log\left(\frac{\ub}{\ub - \mu}\right), \log\left(\frac{1}{\up_{\alpha + \beta}(m)}\right) \right\}.
    \]
Therefore since $\alpha_m + \beta_m \geq 1,$ we have $\up_n(m) \geq 1/(\ualpha + \ubeta)$ and thus
    \[
        \termA \geq - \frac{2 \ubeta }{ \ualpha + \ubeta} \cdot \max\left\{ \log\left( \frac{\ub}{\ub -\mu} \right), \log( \ualpha + \ubeta )\right\} \geq  - \frac{2 \ubeta }{\ualpha + \ubeta} \log \left(\frac{\ub \cdot (\ualpha + \ubeta)}{\ub - \mu}\right).
    \]

    \paragraph*{Term $\termB$}
    Next we analyze the second sum. Since $\supp(\up_\alpha) \subseteq \supp(\up_{\alpha + \beta}) = \supp(q^\star)$, we have
    \[
        \sum_{i: \up_{\alpha + \beta}(i) > 0} \up_\alpha(i)  \log\left( \frac{\up_{\alpha + \beta}(i)}{q^\star(i)} \right) = \sum_{i: \up_\alpha(i) > 0} \up_\alpha(i)  \log\left( \frac{\up_\alpha(i)}{q^\star(i)} \right) + \sum_{i: \up_\alpha(i) > 0} \up_\alpha(i) \log\left( \frac{\up_{\alpha + \beta}(i)}{\up_\alpha(i)} \right)
    \]
    The first term is equal to $\KL(\up_\alpha, q^\star)$ and could be lower-bounded by $\Kinf(\up_\alpha, \mu, f)$ since $q^\star f \geq \mu$. For the second term we analyze the logarithm under condition $\up_\alpha(i) > 0 \iff \alpha_i > 0$. It holds
    \[
        \log\left( \frac{\up_{\alpha + \beta}(i)}{\up_\alpha(i)} \right) = \log\left(  \frac{\alpha_i + \beta_i}{\alpha_i}\right) + \log\left( \frac{\ualpha}{\ualpha + \ubeta} \right) \geq -\log\left(1 + \frac{\ubeta}{\ualpha} \right) \geq -\frac{\ubeta}{\ualpha},
    \]
   and hence
    \[
        \termB \geq \Kinf(\up_{\alpha}, \mu, f) -\frac{\ubeta}{\ualpha}.
    \]
    Finally, we have
    \[
        \Kinf(\up_{\alpha + \beta}, \mu, f) \geq \Kinf(\up_\alpha, \mu, f) - \frac{3 \ubeta}{\ualpha} \log\cdot\left( \frac{\ub \cdot (\ualpha + \ubeta)}{\ub - \mu} \right).
    \]
    For the upper bound we use the variational formula for $\Kinf$ and uniform upper bound (see \cite{honda2010asymptotically}) to obtain
    \begin{align*}
        (\ualpha + \ubeta)\Kinf(\up_{\alpha + \beta}, \mu, f) &=  \max_{\lambda \in [0, 1/(\ub - \mu)]} \sum_{i=0}^m (\alpha_i + \beta_i) \log\left( 1 - \lambda \cdot (f(i) - \mu) \right) \\
        &\leq \max_{\lambda \in [0, 1/(\ub - \mu)]} \sum_{i=0}^m \alpha_i \log\left( 1 - \lambda \cdot (f(i) - \mu) \right)  \\
        &\quad +  \max_{\lambda \in [0, 1/(\ub - \mu)]} \sum_{i=0}^m \beta_i \log\left( 1 - \lambda \cdot (f(i) - \mu) \right)  \\
        &\leq \ualpha \Kinf(\up_\alpha, \mu, f) + \ubeta \log\left( \frac{\ub}{\ub - \mu} \right).
    \end{align*}
\end{proof}

Let $X_1,\ldots,X_n$ be i.i.d. samples from categorical distribution $p \in \simplex_m$. We denote by $\nu_j$ the frequency of a value $j$ : $\nu_j = \sum_{i=1}^n \ind\{ X_i = j\}$ and denote by $\hp_n$ an empirical probability distribution defined as $\hp_n(j) = \nu_j / n$ for all $j \in \{0,\ldots,m\}$. Additionally, let us define $\alpha^0 \in \R^{m+1}$ as a prior counters such that $\sum_{j=0}^m \alpha^0_j = n_0$. Then we define a shifted empirical probability distribution as $\up_n(j) =  \alpha^n_j / (n + n_0)$, where $\alpha^n_j = \alpha^0_j + \nu_j$. In this setup we provide technical results on deviations and concentration for shifted distribution.

\begin{lemma}\label{lem:kinf_concentration}
    Let $f \colon \{0,\ldots,m\} \to [0, \ub]$ be a function such that $f(0) = 0$ and $f(m) = b$. Consider a real number $\mu \in (\up_n f, \ub)$ and define
    \[
        \gamma = \sqrt{\frac{\ub}{\ub - \mu}}\left( 16 \rme^2 + \log^2\left( \frac{\ub}{\ub - \mu} \right) \right).
    \]
    Assume that $\alpha^0_m \geq 1$. Then we have for any $0 < \delta < \gamma / 2$ we have
    \[
        \P\left[ n \Kinf(\up_n, \mu, f) \leq n\left(\Kinf(p, \mu, f) - \delta\right) - 3n_0\log\left( \frac{\ub \cdot (n + n_0)}{\ub - \mu} \right) \right] \leq  \exp(-n \delta^2 / (2\gamma)).
    \]
\end{lemma}
\begin{proof}
    Notice that $\hp_n$ is equal to $\up_\nu$ in terms of Lemma~\ref{lem:distribution_shift} and $\up_n$ is equal to $\up_{\nu + \alpha^0}$. Therefore, we may combine Lemma~\ref{lem:distribution_shift} and Proposition~16 by \cite{garivier2018kl} to conclude the statement.
\end{proof}

Additionally, we state the required deviation result for shifted empirical distributions.
\begin{lemma}\label{lem:kinf_deviation}
    For any $u \geq 0$ we have
    \[
        \P\left[(n+n_0)\Kinf(\up_n, pf, f) \geq n u + n_0 \log\left( \frac{\ub}{\ub - \mu} \right)\right] \leq \rme (2n+1) \rme^{-nu}.
    \]
\end{lemma}
\begin{proof}
    By Lemma~\ref{lem:distribution_shift}
    \[
        (n+n_0) \Kinf(\up_n, \mu, f) \leq n \Kinf(\hp_n, \mu, f) + n_0 \log\left( \frac{\ub}{\ub - \mu} \right)
    \]
    and then we may apply Proposition~14 by \cite{garivier2018kl} to $\Kinf(\hp_n, \mu, f)$.
\end{proof}

\bibliographystyle{plainnat}
\bibliography{ref.bib}

\end{document}